\documentclass[11pt]{amsart}
\usepackage{amsmath,amsfonts,amssymb,graphics,color,enumerate}
\usepackage[applemac]{inputenc}
\usepackage{hyperref}
\usepackage{stmaryrd}
\usepackage{amsthm}
\usepackage{mathrsfs}
\usepackage{mathtools}
\newtheorem{theorem}{Theorem}[section]
\newtheorem{lemma}[theorem]{Lemma}
\newtheorem{proposition}[theorem]{Proposition}
\newtheorem{remark}[theorem]{Remark}

\newtheorem{OpenP}[theorem]{Open problem}

\def\curl{\operatorname{curl}}
\def\div{\operatorname{div}}
\providecommand{\R}{\mathbb{R}}

\providecommand{\N}{\mathbb{N}}

\providecommand{\varepsilon}{}
\renewcommand{\leq}{\leqslant}
\renewcommand{\geq}{\geqslant}

\numberwithin{equation}{section}
\begin{document}
\date{\today}
\title[Remote trajectory tracking of rigid bodies in a $2$D perfect incompressible fluid]{Remote trajectory tracking of rigid bodies \\ immersed in a $2$D perfect incompressible  fluid}
%
%
\author{Olivier Glass} 
\address{CEREMADE, UMR CNRS 7534, Universit\'e Paris-Dauphine, PSL Research University,
Place du Mar\'echal de Lattre de Tassigny, 75775 Paris Cedex 16, France} 
\author{J\'ozsef J. Kolumb\'an} 
\address{Institut f\"ur Mathematik, Universit\"at Leipzig, D-04109, Leipzig, Germany} 
\author{Franck Sueur}
\address{Institut de Math\'ematiques de Bordeaux, UMR CNRS 5251,
Universit\'e de Bordeaux, 351 cours de la Lib\'eration, F33405 Talence Cedex, France
$\&$ Institut  Universitaire de France}
%
%
\begin{abstract}
We consider the motion of several rigid bodies immersed in a two-dimensional incompressible perfect fluid. The motion of the rigid bodies is given by the Newton laws with forces due to the fluid pressure and the fluid motion is described by the incompressible Euler equations. 
Our analysis covers the case where the circulations of the fluid velocity around the bodies are nonzero and where the fluid vorticity is bounded. 
The whole system occupies a bounded simply connected domain with an external fixed boundary which is impermeable except on an open non-empty part where one allows some fluid to go in and out the domain by controlling the normal velocity and the entering vorticity. 
We prove that it is possible to exactly achieve any non-colliding smooth motion of the rigid bodies by the remote action of a controlled normal velocity on the outer boundary which takes the form of state-feedback, with zero entering vorticity. 
This extends the result of 
 (Glass, O., Kolumb\'an, J. J., Sueur, F. (2017). External boundary control of the motion of a rigid body immersed in a perfect two-dimensional fluid. Analysis \& PDE)   
where the exact controllability of a single rigid body immersed in a 2D irrotational perfect incompressible fluid 
 from an initial position and velocity to a final position and velocity was investigated. 
The proof relies on a nonlinear method to solve 
linear perturbations of nonlinear equations associated  with a quadratic operator having a regular non-trivial zero.
Here this method is applied to a quadratic equation satisfied by a class of boundary controls, which is 
obtained by extending  the reformulation of the Newton equations   performed in the uncontrolled case in (Glass, O., Lacave, C., Munnier, A., Sueur, F. (2019). Dynamics of rigid bodies in a two dimensional incompressible perfect fluid. Journal of Differential Equations, 267(6), 3561--3577) to the case where a control acts on the external boundary. 
\end{abstract}
%
%
\maketitle
%
%
%
%
%
%
%
%
%
%
%
%
%
\section{Presentation of the model: the ``Euler+rigid bodies'' system}
The model that we consider in this paper describes the motion of rigid bodies immersed in a two-dimensional perfect incompressible fluid. The whole system occupies a bounded connected open subset  $\Omega$ of $\R^2$, which to simplify we will also consider to be simply connected (though this is by no means essential to the analysis).
The rigid bodies occupy at the initial time disjoint non-empty regular connected and simply connected compact sets $\mathcal S_{ \kappa,0} \subset \Omega$, with $\kappa $ in $ \{1,2, \ldots , N \}$.
We assume for simplicity that none of these sets is a disk, since this particular case requires a special treatment. 
 
The rigid motion of the solid $\kappa$ is described at each moment by the rotation matrix
\begin{equation*}
R(\theta_\kappa (t)) := 
\begin{bmatrix}
\cos \theta_\kappa (t) & - \sin \theta_\kappa (t) \\
\sin \theta_\kappa (t) & \cos \theta_\kappa (t)
\end{bmatrix}, \quad \theta_\kappa (t) \in \R,
\end{equation*}
and by the position $h_\kappa(t) $ in $ \mathbb R^2$ of its center of mass.  
The domain of the solid $\kappa$ at every time $t>0$ is therefore 
$$\mathcal S_\kappa(t) :=R(\theta_\kappa(t))(\mathcal S_{ \kappa,0}-h_{\kappa}(0)) +h_\kappa(t). $$ 
We will denote by $m_\kappa>0$ and by $\mathcal{J}_\kappa>0$  respectively the mass and the moment of inertia of the body indexed by $\kappa$. 
The domain occupied by the fluid is correspondingly
\begin{equation*}
\mathcal F_0 := \Omega \setminus \bigcup_{ \kappa \in \{1,2, \ldots , N \} } \, {\mathcal S}_{\kappa,0} ,  \ \text{ at } t=0 , \ \text{ and } 
\mathcal F(t) := \Omega\setminus \bigcup_{ \kappa \in \{1,2, \ldots , N \} } \, {\mathcal S_\kappa}(t)  \ \text{ at } t>0. 
\end{equation*}
We will denote by $u=(u_1,u_2)^t$ (the exponent $t$ denotes the transpose of the vector)  and by $\pi$ the velocity and pressure fields in the fluid, respectively. Without loss of generality, the fluid is supposed to be homogeneous of density $1$, to simplify the notations. 
The fluid dynamics is given by the incompressible Euler equations:
\begin{gather}
\label{E1}
\frac{\partial u}{\partial t}+(u\cdot\nabla)u +\nabla \pi=0\quad\text{in }\mathcal F(t), \quad \text{for } t >0, \\
\label{E2}
\div u=0 \quad \text{in }\mathcal F(t), \quad \text{for } t >0 . 
\end{gather}
The solids dynamics is given by Newton's balance law for linear and angular momenta:  given $ \kappa $ in $ \{1,2, \ldots , N \}$, 
\begin{gather}
\label{EqTrans}
m_\kappa h_\kappa'' (t)=\int_{\partial\mathcal S_\kappa(t)}\pi n\, {\rm d}s,  \quad \text{for } t >0,  \\
\label{EqRot}
\mathcal J_{\kappa} \theta_\kappa'' (t)
=\int_{\partial\mathcal S_\kappa(t)}(x-h_\kappa(t))^\perp\cdot \pi n\, {\rm d}s,  \quad \text{for } t >0. 
\end{gather}
When $x=(x_1,x_2)^t$ the notation $x^\perp $ stands for $x^\perp =( -x_2 , x_1 )^t$,  $n$ denotes the unit normal vector on $\partial \mathcal{S}_\kappa(t)$ which points outside of the fluid, so that $n = \tau^\perp$, where $\tau$ is the unit counterclockwise tangential vector on $\partial \mathcal S_\kappa (t)$. 
We assume the rigid bodies to be impermeable so that we prescribe on the interface: for every $ \kappa $ in $ \{1,2, \ldots , N \}$,
\begin{equation} \label{souslab}
u \cdot n = \big( \theta'_\kappa (\cdot-h_\kappa)^\perp + h'_\kappa \big) \cdot n  \quad  \text{on }\partial\mathcal S_\kappa(t),  \quad \text{for } t >0 .
\end{equation}
We will use the notations $\mathbf{q}_{\kappa}$ and $\mathbf{q}'_{\kappa}$ for vectors in $\R^{3}$ gathering both the linear and angular parts of the position and velocity:
$$\mathbf{q}_{\kappa} := (h_{\kappa}^t,\theta_{\kappa})^t \quad \text{ and } \quad \mathbf{q}_{\kappa}':=(h^{\prime\, t}_{\kappa},\theta'_{\kappa})^t .$$
The vectors $\mathbf{q}_{\kappa} $ and $\mathbf{q}'_{\kappa} $ are next concatenated into vectors of length $3N$:
$$q = (\mathbf{q}_{1}^t , \ldots , \mathbf{q}_{N}^t)^t \quad \text{ and } \quad{q}' = (\mathbf{q}_{1}^{\prime\, t} , \ldots , \mathbf{q}^{\prime\, t}_{N} )^t, $$
whose entries are relabeled respectively $q_k$ and $q'_k$ with $k$ ranging over $\{1,\ldots,3N\}$. Hence we have also:
$$ q = (q_1, q_2,\ldots,q_{3N})^t \quad \text{ and } \quad {q'} = (q'_1,q'_2,\ldots,q'_{3N})^t.$$
Consequently, $k $ in $ \{1,\dots,3N\}$ denotes the datum of both a solid number and a coordinate in $\{1,2,3\}$ so that $q_{k}$ and $q'_{k}$ denote respectively the coordinate of the position and of the velocity of a given solid. 
More precisely, for all $k $ in $ \{1,\dots,3N\}$, we denote by $\llbracket k \rrbracket$ the quotient of the Euclidean division of $k-1$ by $3$,
$[k] = \llbracket k \rrbracket +1 $ in $ \{1,\dots,N\}$ denotes the number of the solid and $(k) := k - 3\llbracket k \rrbracket $ in $ \{1,2,3\}$ the considered coordinate. 

Throughout this paper we will not consider collisions, so we introduce the set of body positions without collision:
\[
\mathcal{Q} := \{q \in \R^{3N} \ :\ \min_{\kappa \neq \nu}{\rm d}(\mathcal{S}_{\kappa}(q), \Omega^c \cup \mathcal{S}_{\nu}(q)) > 0\},
\]
where $d$ is the Euclidean distance. For $\delta>0$, we also introduce
\[
\mathcal{Q}_\delta := \{q \in \R^{3N} \ :\ \min_{\kappa \neq \nu}{\rm d}(\mathcal{S}_{\kappa}(q), \Omega^c \cup \mathcal{S}_{\nu}(q)) \geq \delta\}.
\]

The fluid domain is completely described by $q$ in $ \mathcal{Q}$ and we will therefore make use of the following abuse of notation: $\mathcal{F}(t)= \mathcal{F}(q(t))$.

%
%
%
%
%
%
%
%
%
%
%
%
%
%
\section{Boundary conditions on the external boundary}

Our purpose in this paper is to investigate the possibility of steering the rigid bodies according to any reasonable (smooth, non colliding) given motion
by means of a boundary control acting on a part of the external boundary, while on the rest of the boundary we consider the usual impermeability condition. 
More precisely we consider  $\Sigma$ a nonempty, open part of the outer boundary $\partial \Omega$ and the following 
 boundary conditions introduced by Yudovich in \cite{Yudo}. To begin with, let $\mathcal C$ denote the space
 \begin{equation}
 \label{def-Cpasb}
 \mathcal C := \left\{  g \in    C_{0}^{\infty}( \Sigma  ;\mathbb{R}) \  \text{ such that }  \,  \int_\Sigma g \, {\rm d}s=0 \right\} .
 \end{equation}
For $T>0$, we consider ${g}$ in $ C^{\infty}([0,T]; \mathcal C )$ and the boundary condition on the normal trace of the outer boundary
\begin{equation} \label{Yudo1}
u(t,x)\cdot n(x)=g(t,x)\ \text{on}\ [0,T]\times\Sigma \quad \text{ and } \quad u(t,x)\cdot n(x) =0\ \text{on}\ [0,T] \times (\partial \Omega \setminus \Sigma) .
\end{equation}
Above, as for the solids boundaries, $n$ denotes the unit normal vector pointing outside the fluid, so that $n = \tau^\perp$, where here $\tau$ denotes the unit clockwise tangential vector on $\partial \Omega$.
The condition on the zero flux of $g$ through $\Sigma$ is necessary due to the incompressibility of the fluid. As noticed by Yudovich (Ibid.), this is not a sufficient boundary condition to determine the system. To complete it, we consider the set
$$\Sigma^- := \{(t,x)\in [0,T]\times\Sigma \  \text{ such that }  \ g(t,x)<0\}, $$
of points of $[0,T] \times \Sigma$ where the fluid velocity field points inside $\Omega$. 
Then the other part of the boundary condition consists in prescribing the entering vorticity, that is the vorticity 
$$\omega := \curl u =\partial_1 u_2- \partial_2 u_1\text{ on }\Sigma^-. $$ This is natural since the fluid vorticity satisfies the transport equation:
\begin{equation} \label{transp-vort}
\frac{\partial \omega }{\partial t}+(u\cdot\nabla) \omega =0 , \quad x \in \mathcal{F}(q(t)) .
\end{equation}
For simplicity we will actually prescribe a null control in vorticity, that is
\begin{equation} \label{Yudo2}
\omega (t,x)=0\ \text{on}\ \Sigma^{-} .
\end{equation}

Let us insist on the fact that the control considered here is a remote control in the sense that it is located on the external boundary, not on the moving rigid bodies. For this alternative issue of rigid or deformable bodies equipped with thrusters or locomotion devices we refer to the papers \cite{GR,LR1,LR2,Munnier:2008ab}. \par
Since the fluid occupies a multiply-connected domain, the circulations of the fluid velocity around the rigid bodies $\mathcal S_{\kappa}$:
\begin{equation}\label{eq:circs}
  \int_{\partial\mathcal S_\kappa (t)} u(t)\cdot\tau \, {\rm d}s = \gamma_\kappa , \quad  \text{ for all } \ \kappa  \in \{1,2, \ldots , N \},
\end{equation}
will play an important role.
Let us recall that, for each $\kappa$, the circulation $\gamma_\kappa$ remains constant over time according to Kelvin's theorem. We will use the notation 
$$
\gamma:= (\gamma_\kappa )_{\kappa=1,\ldots,N} .
$$ 
To achieve our goal, we will consider a control {\it in feedback form}, depending on the state of the  ``fluid+rigid bodies'' system. 
More precisely we will prescribe a normal velocity $g$ on $[0,T] \times \Sigma$ of the form
\begin{equation} \label{Feedback}
g(t)  =  \mathscr C  (q(t),q'(t),q''(t),\gamma,\omega(t,\cdot)),
\end{equation}
where $\mathscr C$ is a Lipschitz function on 
$$\cup_{q\in\mathcal Q_\delta} \{ q \} \times \R^{3N} \times \R^{3N} \times \R^{N} \times L^\infty ( \mathcal{F}(q); \R) ,$$
 for any  $\delta >0$.  
Furthermore we will only need a finite dimensional space of controls so that  $\mathscr C$ can be taken with values in a finite dimensional subspace of the space $\mathcal C$ defined in \eqref{def-Cpasb}.
%
%
%
%
%
%
%
%
%
%
%
%
%
%
\section{Main results: Trajectory tracking by a remote control}
\label{sec-mr}
The problem that we raise in this paper is the trajectory tracking by means of the remote control described in the previous section. Precisely, the question is: is it possible to exactly achieve any non-colliding smooth motion of the rigid bodies by the remote action described above? The purpose is schematically described in Figure~\ref{Fig:1}. Let us now explain how we positively answer to this question. \par

\begin{figure}[ht]
\begin{center}
	\scalebox{1.5}{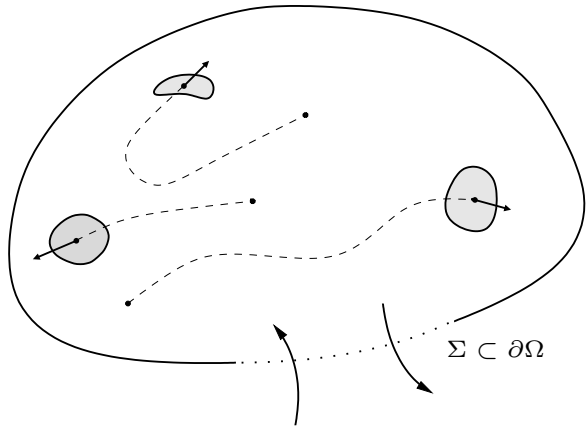}
\end{center}
\caption{Controlled trajectories for solids inside $\Omega$}
\label{Fig:1}
\end{figure}
To begin with, let us be more specific on the functional setting.
Following Yudovich \cite{Yudo2}, we consider the case where the initial fluid vorticity is bounded.
Then the natural regularity for a  fluid velocity field associated with a bounded vorticity is the log-Lipschitz regularity. 
Precisely, for $T>0$ and given the solids trajectories $q$, we will consider the space $LL(T)$ of uniformly in time log-Lipschitz in space vector fields, defined via its norm
$$
\| f \|_{\text{LL(T)}} := \|f\|_{L^\infty(\cup_{t \in (0,T)}  \, {t}  \times \mathcal{F} (q(t) ))}
+ \sup_{t \in [0,T]} \sup_{x\neq y} \frac{|f(t,x)-f(t,y)|}{|x-y|(1+\text{ln}^{-}(|x-y|))}.
$$
Moreover, we will work with vorticities belonging to balls in $L^\infty$: for any $q$ in $\mathcal Q_\delta$ and $r_\omega>0$, we consider the complete metric space 
\begin{align} \label{lesboules}
\mathscr{B}(q,r_\omega):=\overline B_{L^\infty(\mathcal{F}(q))} (0,r_\omega) 
\text{ endowed with the } L^3(\mathcal{F}(q)) \text{ distance}. 
\end{align}
Before we state the main result of this paper, let us give two words of caution.
\begin{itemize}
\item 
Below we will use the letter $q$ as a variable for the positions of the rigid bodies, as well as a trajectory of the rigid bodies. Readers should not be confused. 
\item 
Let us also recall that, in incompressible fluid mechanics, including in the presence of moving rigid bodies, the pressure field $ \pi$  can be interpreted as a Lagrange multiplier associated with the divergence-free constraint;
as a result it can be ignored when we speak of a solution of the problem.
Consequently in the sequel we will say that  $(q,u)$ satisfies, for $t$ in $[0,T]$, 
the Euler equations \eqref{E1}-\eqref{E2} and 
the Newton equations \eqref{EqTrans}-\eqref{EqRot} for  $ \kappa $ in $ \{1,2, \ldots , N \}$,
without referring to the associated pressure.
In the case where the vorticity is bounded, the controlled solutions which we will consider below correspond to a pressure field in  $L^{\infty}(0,T; H^{1}({\mathcal F}(t)))$  
which is unique up to a function depending only on time which does not change the value of the terms involving the pressure  \eqref{E1}, \eqref{EqTrans} and \eqref{EqRot}.
This regularity result can be obtained as in the uncontrolled case, see  \cite[Corollary 2]{GS-Uniq}. 
In particular this gives a sense to the right hand sides of  the Newton equations \eqref{EqTrans}-\eqref{EqRot} for  $\kappa$ in $\{1,2, \ldots , N \}$.
\end{itemize}
Our main result is twofold. In a first part, we prove that there exists a feedback control $\mathscr C$ as in \eqref{Feedback} such that, for any target trajectory $q$ and any compatible initial conditions, there exists a solution of the closed-loop system with this control $\mathscr C$, in which the solids follow the trajectory $q$ exactly. The compatible initial conditions can be described solely by the initial vorticity $\omega_{0}$ and the circulations $\gamma$, the rest being imposed by $q(0)$, $q'(0)$ and ${\mathscr C}$. The second part of our statement establishes a partial uniqueness result: any (weak) solution of the above-closed loop system does satisfy that the solids follow the trajectory $q$ exactly.
The central idea is that the control imposes the correct acceleration. \par
The exact statement is as follows.
\begin{theorem} \label{main}  
For any $\delta >0$, there is  a  finite dimensional subspace $\mathcal E$ of $\mathcal C$ such that the following holds.
Let $T>0$, $r_\omega>0$ and $\mathscr K$ be a compact subset of  $  \R^{3N}  \times   \R^{3N}   \times   \R^{N}$. Then there exists a control law
$$
\mathscr C \in \text{Lip}( \cup_{q\in \mathcal Q_\delta} \{  q\} \times \mathscr K  \times \mathscr{B}(q,r_\omega); \mathcal E),
$$
such that the two following results hold true  for  any given trajectory  $q$ in $C^{2} ([0,T] ; \mathcal Q_\delta)$
and for any $\gamma$ in $\R^{N}$ such that for any  $t $ in $ [0,T]$,  $(q'(t),q''(t),\gamma)$ belongs to $\mathscr{K}$.  \par
\ \par
\begin{enumerate}[1)]
\item 
For any initial vorticity  $\omega_0$ in $L^\infty ( \mathcal{F} (q (0)))$ such that $\|\omega_0\|_{L^\infty(\mathcal{F}_0)}\leq r_\omega$, 
there exists a velocity field  $u $ in $   LL(T) \cap C^0([0,T];W^{1,p}(\mathcal{F}(t)))$, for all $p\in [1,+\infty)$, with  $\curl u (0,\cdot) = \omega_0$ and for any
$t$ in $[0,T]$, $\curl u (t,\cdot) $ in $ \mathscr{B}(q(t),r_\omega)$, 
such that $(q,u)$ satisfies for all $t$ in $[0,T]$:
the Euler equations \eqref{E1}-\eqref{E2},
the Newton equations \eqref{EqTrans}-\eqref{EqRot} for $\kappa$ in $ \{1,2, \ldots , N \}$,  
the interface condition \eqref{souslab}, 
the boundary condition \eqref{Yudo1} on the normal velocity with 
\begin{equation} \label{pasmixe}
g(t)= \mathscr C (q(t),q'(t),q''(t), \gamma,\curl u(t,\cdot)),
\end{equation}
and the boundary condition \eqref{Yudo2} on the entering vorticity
and the circulation conditions \eqref{eq:circs}.  \par 
\par \ \par 

\item  Let $$(\tilde{q},\tilde{u}) \in C^{2} ([0,T] ; \mathcal Q_\delta ) \times [LL(T)\cap C^0([0,T];W^{1,p}(\mathcal{F}(t)))], \text{ for all } p\in [1,+\infty),$$ 
 and $\tilde{\gamma}$ in $\R^{N}$ such that 
for any $t$ in $ [0,T]$,  $(\tilde{q}'(t),\tilde{q}''(t),\tilde{\gamma})$ belongs to $\mathscr{K}$ and $\curl \tilde{u} (t,\cdot) $ is in $ \mathscr{B}(\tilde{q}(t),r_\omega)$.
Assume that $(\tilde{q},\tilde{u})$ satisfies: 
the Euler equations \eqref{E1}-\eqref{E2}, 
the Newton equations \eqref{EqTrans}-\eqref{EqRot} for $\kappa $ in $ \{1,2, \ldots , N \}$,
the interface condition \eqref{souslab},
the boundary condition \eqref{Yudo1} on the normal velocity with  
\begin{equation} \label{mixe}
g(t)= \mathscr C (\tilde{q}(t),\tilde{q}'(t), q''(t),\tilde{\gamma},\curl \tilde{u}(t,\cdot)) ,
\end{equation}
the boundary condition \eqref{Yudo2} on the entering vorticity,
the circulation conditions \eqref{eq:circs} (with $\tilde{\gamma}$ in place of ${\gamma}$)
and the initial conditions  $\tilde{q}(0)= q(0)$ and  $ \tilde{q}'(0) = q'(0) $ on the initial positions and velocities of the rigid bodies. 
Then  $\tilde{q} =q$ on $[0,T]$. 
\end{enumerate}
\end{theorem}
We note that there is a slight abuse of notation in writing $C^0([0,T];W^{1,p}(\mathcal{F}(t)))$, for space of functions defined for each $t$ in the fluid domain $\mathcal{F}(t)$.
Furthermore, we require this added regularity of the velocity field to insure that the trace is well-defined at $t=0$.
\ \par 
A few further comments are in order. \par
\ \par
\noindent 
\textbf{Comparison with the controllability result  in \cite{GKS}.}
Theorem~\ref{main} extends the result in \cite{GKS} where the exact controllability of a single rigid body immersed in a 2D irrotational perfect incompressible fluid 
 from an initial position and velocity to a final position and velocity was investigated. There the control was already set on a non-empty open part of the external boundary and was obtained as a regularization of some time impulses. 
On the opposite  Theorem~\ref{main}  proves that it is possible to drive some rigid bodies along a given admissible trajectory by a control which is active all  the time, while, in terms of the space variable, this control is also supported on  a non-empty open part of the external boundary. Moreover the control in  Theorem~\ref{main}  has the convenience to be achieved as a feedback law, depending only on the instantaneous state of the fluid-rigid bodies system. Thus Theorem~\ref{main}  provides a positive answer to the open problem mentioned in \cite{GKS}, in the wider setting where several rigid bodies and irrotational flows are considered. Of course Theorem~\ref{main}
 also allow implies controllability of the positions and velocities of the rigid bodies at final time, by considering a targeted trajectory with the desired final positions and velocities. 
 For instance, in view of practical applications, one may think at a regrouping of the rigid bodies in a given subregion of the domain, with enough volume to contain them all with positive distances. In the opposite direction one may think at a spreading of the rigid bodies in the fluid domain, thinking at  a medical treatment which requires dispersion of some medicinal particles. \par
Hence this provides an extension to the main result of \cite{GKS}, but on the other hand, since the control can be active all the time in the result of Theorem~\ref{main}, it is not possible to guarantee a small total flux condition as we did in \cite[Remark 1]{GKS} by a simple rescaling in time. 
\par
\ \par
\noindent
\textbf{Uniqueness part (second part) of Theorem~\ref{main}.}
In the case of an $L^\infty$ vorticity, a uniqueness result for the fluid-solid system has been obtained in \cite{GS-Uniq} in the case  without control, that is, of impermeable boundary condition (vanishing normal component) on the whole external boundary, rather than the permeable boundary conditions \eqref{Yudo1} and \eqref{Yudo2}. 
For the latter, uniqueness in the setting of bounded vorticity is a delicate issue, already in the case of a fluid alone. Indeed, in contrast to his celebrated result in the impermeable case \cite{Yudo2}, Yudovich only succeeded to obtain uniqueness for solutions which are much more regular in \cite{Yudo}. Recently  Weigant and Papin obtained  in  \cite{WP}  the uniqueness of the solutions with bounded vorticity with a proof in the case  of a rectangle with the flow entering on a lateral side and exiting on the opposite side. 
The extension of such a uniqueness result to the case of the fluid-rigid bodies system seems challenging, as it involves a free boundary problem and  a more involved geometry. Hence this leaves the following open problem.
\begin{OpenP} \label{OP!}  
For any $\delta >0$, for any $q_0$ in $\mathcal Q_\delta$ and $q_1$ in $ \R^{3N}$, 
for any initial vorticity  $\omega_0$ in $L^\infty ( \mathcal{F} (q (0)))$ 
for any $g$ in  $ C^{\infty}([0,T]; \mathcal C )$, 
there exists $T>0$, there exists a unique velocity field  $u $ in $   LL(T)\cap C^0([0,T];W^{1,p}(\mathcal{F}(t)))$, for all $p\in [1,+\infty)$, with  $\curl u (0,\cdot) = \omega_0$ and for any
$t$ in $[0,T]$, $\curl u (t,\cdot) $ in $L^\infty(\mathcal{F}(q))$, and a unique $q$ in $C^{2} ([0,T] ; \mathcal Q_\delta )$
with the initial conditions  $q(0)= q_0$ and  $ q'(0) = q_1 $ on the initial positions and velocities of the rigid bodies, 
such that $(q,u)$ satisfies for all $t$ in $[0,T]$:
the Euler equations \eqref{E1}-\eqref{E2},
the Newton equations \eqref{EqTrans}-\eqref{EqRot} for $\kappa$ in $ \{1,2, \ldots , N \}$,  
the interface condition \eqref{souslab}, 
the boundary condition \eqref{Yudo1} on the normal velocity, 
the null boundary condition \eqref{Yudo2} on the entering vorticity
and the circulation conditions \eqref{eq:circs}. 
\end{OpenP}

However the second part of Theorem~\ref{main} claims that, for a given bounded initial vorticity, 
 should there be several solutions, the control would drive the rigid bodies of all these solutions along the targeted motion. On the other hand the fluid motions are not guaranteed to coincide. 

 Let us emphasize that the control in  \eqref{mixe} involves the acceleration  $q''(t)$ of the targeted motion, rather than the acceleration $\tilde{q}''$ of the solution itself  as in \eqref{pasmixe} for the part of the statement regarding the existence of a motion associated with the targeted trajectory $q$. 
Indeed the result in the  second part of Theorem~\ref{main} is more general since it covers the case of solutions corresponding to distinct initial vorticities and velocity circulations. 
On the other hand if the initial positions and velocities of the rigid bodies do not match, that is if $(\tilde{q}(0),\tilde{q}'(0)) \neq (q(0),q'(0) )$, 
 the control law \eqref{pasmixe} will not guarantee any decay of the initial condition errors. 
 Still the control law  can be adapted to provide a stability result 
 in the case where the sole assumption on the  initial conditions is that  $\tilde{q}(0)$ and $ q(0)$ are sufficiently close and that 
 the boundary condition \eqref{Yudo1} on the normal velocity is satisfied with  
\begin{equation} \label{mixe-stab}
g(t)= \mathscr C (\tilde{q}(t),\tilde{q}'(t), q''(t) + K_P (q(t) - \tilde{q}(t)) + K_D   (q'(t) - \tilde{q}'(t)) ,\tilde{\gamma},\curl \tilde{u}(t,\cdot)) ,
\end{equation}
 where  $K_P$ and $K_D$  are  positive definite, symmetric $3N  \times  {3N} $ matrices. 
Then the error $q (t) - \tilde{q}(t)$ exponentially decays to $0$ as the time $t$ goes to $+\infty$, with a rate which can be made arbitrarily fast by appropriate choices of $K_P$ and $K_D$. 
The  indexes $P$ and $D$ in the notations $K_P$ and $K_D$ respectively refer to ``proportional'' and ``derivative'', according to the usual terminology in robotics, see  \cite[Chapter 11]{MLS}, \cite[Section 4.5]{SHV} and  \cite[Section 8.5]{SSVO}. This stability result will be proved by a little modification of the proof of the second part of Theorem~\ref{main}. 
The interest of such a result is that in practice the positions and velocities of the rigid bodies cannot be determined exactly, see \cite{CCOR,CMM,MR}. 
\par
\ \par
\noindent
\textbf{Regularity issues.}
A natural issue is whether it is possible to preserve the 
 regularity of the vorticity when time proceeds 
when achieving the targeted motion. 
In that case it could be possible to  adjust the boundary condition on the entering vorticity, by substituting to \eqref{Yudo2} an appropriate inhomogeneous condition. 
Such a construction was performed in \cite{Coron-Reg}  in the case of the controllability of a fluid alone by means of an open-loop control.

On the opposite direction, one may wonder whether it is possible to extend the existence part of Theorem~\ref{main} to the case where the initial  vorticity is only in $L^p$ with $p\geq 1$. Let us recall in that direction that, on the one hand, some existence results, in the case of impermeable  boundary condition,  rather than the  conditions \eqref{Yudo1} and \eqref{Yudo2}, have been obtained in \cite{GS-i,GS-l,GLS,Sueur,WX} in the case of systems coupling one single rigid body and a two-dimensional perfect incompressible flow without any external boundary; and, on the other hand, some existence results have been obtained  in \cite{BS,MU,M} in the case of boundary conditions such as \eqref{Yudo1} and \eqref{Yudo2}, but for a fluid alone. 

One may also wonder whether it is possible to reach some targeted  trajectories with lower regularity in time, by the means of controls $g$ which are also of lower regularity in time. Indeed this is very much related to the strategy of  \cite{GKS} where impulsive controls were considered. 
\par
\ \par
\noindent
\textbf{Energy saving.}
A natural question is whether it is possible to turn off the control when the targeted motion is, at some time, solution of the uncontrolled equation, that is to guarantee that the mapping $\mathscr C$ vanishes on 
$(q,q',q'',\omega)$ satisfying  \eqref{E1}- \eqref{transp-vort},  \eqref{EqTrans}-\eqref{EqRot} for  $ \kappa $ in $ \{1,2, \ldots , N \}$, \eqref{souslab},  \eqref{Yudo1} with $g= 0$ and \eqref{Yudo2}. Some extra comments on this issue are given in Section~\ref{afaire}, after the proof of Theorem~\ref{main}; this postponing allows us to be more precise regarding some technical aspects of the question.
\par
\ \par
\noindent
\textbf{Three dimensional case.}
Another natural extension of the results of Theorem~\ref{main} is the case where the system is set in three space dimensions. In the impermeable/uncontrolled case a reformulation of the Newton equations as a second-order ODE for the solid positions is tackled  in \cite{GS3D}. However the design of the control below relies on the possibility to use complex analysis in two dimensions. Therefore several arguments of the proof of Theorem~\ref{main} regarding the construction of the control law would need to be adapted. 
\par
\ \par
\noindent
\textbf{Controlled collisions.}
A challenging question is whether it is to possible to provoke some controlled collisions. Let us recall that collisions can occur even without control, see \cite{H,HM,C}. However one may imagine to be able,  given a couple of rigid bodies, some collision positions ---and perhaps also some collision velocities--- to prove the existence of a control such that for the corresponding controlled solution of  \eqref{E1}- \eqref{transp-vort},  \eqref{EqTrans}-\eqref{EqRot} for  $ \kappa $ in $ \{1,2, \ldots , N \}$, \eqref{souslab}, this collision occurs.  
In this direction let us mention that in  the proof of Theorem~\ref{main} 
we will use certain arguments of  complex analysis, see also \cite{GKS}, which involve approximations by rational functions  of harmonic functions which are first defined in some respective neighborhoods of the rigid bodies. 
When two rigid bodies become close one could face some difficulty in the fusion of such local approximations, see  \cite{Gaier}. 
\par
 \ \par
\noindent
\textbf{Case of the Navier-Stokes equations.}
One could also be interested in extending the results of Theorem~\ref{main} 
 to the  case of the Navier-Stokes equations rather than the Euler equations as a model for the fluid part of the system. Then the boundary conditions have to be modified and several choices can be made.
Beside the classical no-slip conditions, there is a condition which models slip and friction at the boundaries referred to here as Navier slip-with-friction boundary conditions. The latter case is  closer to the case of the Euler equation where slip, that is  discrepancy of the tangential velocity at the boundary, is allowed. 
 Regarding the Cauchy problem, in the uncontrolled case, 
 the existence of weak Leray-type solutions to the Navier-Stokes system in presence of a rigid body when the Navier slip-with-friction conditions are considered at the boundaries 
has been  proved in \cite{PS} when the system occupies the whole space and in \cite{GVH} 
 when the system occupies a bounded domain.
These two results tackle the three-dimensional case but the latter result has been adapted in \cite{B} to the two-dimensional case with some extra properties.
For strong solutions, existence and uniqueness of solutions in some Hilbert spaces have been proved in  \cite{Wang}.
 
On the other hand the  result in \cite{GKS}, on the exact controllability of a single rigid body immersed in a 2D irrotational perfect incompressible fluid   mentioned above was extended in \cite{K} to the Navier-Stokes equations in the case where the Navier slip-with-friction boundary conditions are prescribed on the interface between the fluid and the body,
see also \cite{D,DT} for complementary results.    One key ingredient was a rescaling in time 
 which allows to reduce the problem to the case where the viscosity is small (first introduced by Coron in \cite{Coron-NS}), to use an asymptotic expansion and the inviscid  result of \cite{GKS} for the leading order. This strategy works for the Navier slip-with-friction boundary conditions  because the corresponding boundary layers have a small amplitude. 
 Unfortunately, for the problem of trajectory tracking considered in this paper,  one is not allowed to effectuate such a time-rescaling. 
 
 However it is possible that the method used in the proof of Theorem~\ref{main} could be adapted to the case 
  the Navier-Stokes equations, with the Navier slip-with-friction  boundary conditions. 
The case of  the Navier-Stokes equations, with the no-slip  boundary conditions is clearly more challenging. 
  \par
 \ \par
\noindent
\textbf{A nonlinear method reminiscent of Coron's return method.}
To prove Theorem~\ref{main} we will make use of a nonlinear method  which is reminiscent of Coron's return method, cf. \cite[Chapter 6]{Coron}, in the sense that it takes advantage of the nonlinearity of the problem  out of equilibria.
 However our method rather considers time as a parameter and allows us to prove a trajectory tracking result rather than a controllability result.  
It uses the homogeneity of the nonlinear part of a nonlinear equation and the existence of a single non-trivial zero at which the differential of this term is right-invertible to solve the equation for general data, see Section  \ref{sec-proto}. One could also compare to Coron's Phantom tracking method from \cite{Coron-Pha}, which takes advantage of the nonlinearity in a similar fashion in order to establish a stabilization result.
  \par
 \ \par
\noindent
\textbf{Practical use.}
An attempt to put in practice the theoretical result in Theorem~\ref{main} would face the drawback of the feedback laws  \eqref{mixe} and  \eqref{pasmixe} depend on the full state-function $\omega(t,\cdot)$, rather than on only some norms, moments or any finite dimensional information extracted from it. However the Lipschitz dependence leads to the hope that a bad identification of the vorticity of the fluid by the operator in charge to apply the control at the boundary may not affect the resulting controlled trajectory too drastically. Another difficulty is linked to the design of this control law by itself. Indeed, in this direction, for a quite important part of the analysis performed below in the proof of Theorem~\ref{main} the observations done in 
\cite{HK} for a slightly different problem are also relevant. There the authors discuss an alternative method to the complex-analytic one which is developed here. This method is more application-friendly. However this alternative method relies on linear techniques which seem difficult to adapt here.

\section{Organisation of the rest of the paper.}
The rest of the paper is organised as follows. 
In Section~\ref{sec:no} we recall the decomposition of the fluid velocity into elementary velocities according to the vorticity, the circulations, the external boundary control and the velocities of the rigid bodies. 
Then we reformulate  in Section~\ref{sec:no2}  
the solid equations as an equation with the control as the unknown, and the solid motion and the vorticity as data. 
In Section~\ref{sec:control} we design the feedback control. Section~\ref{sec:existence} is devoted to the end of the proof of the first part of Theorem~\ref{main} regarding the existence of a controlled solution with the targeted motion of the rigid bodies. 
Then, Section~\ref{sec:conclusion-r} is  devoted to the end of the proof of the second part of Theorem~\ref{main} regarding the uniqueness of the motion of the rigid bodies for the 
hybrid control law   \eqref{mixe}. 
Finally  in Section  \ref{afaire} we give some extra comments on the issue of energy saving discussed above.

%
%
%
%
%
%
%

%
%
%
%
\section{Decomposition of the fluid velocity according to the solids motions, the vorticity, the circulation and the external control}
\label{sec:no}
Let $q$ in $ \mathcal{Q}$. For any $q'$ in $\R^{3N}$, for any $\omega$ bounded over $\mathcal{F}(q)$,  for any $\gamma:= (\gamma_\kappa )_{\kappa=1,\ldots,N} $ in $\R^{N}$, for any $g$ in $\mathcal C$, classically
there exists a unique  log-Lipschitz vector field $u$  such that 
\begin{subequations} 
\label{div-curl}
\begin{gather}
\div u =0 \text{ in } \mathcal{F} (q), \quad
\curl u =\omega \text{ in } \mathcal{F} (q), \quad u\cdot n =g  \text{ on } \partial\Omega, \\
u \cdot n = \big( \theta^{\prime}_\kappa (\cdot-h_\kappa)^\perp + h^{ \prime}_\kappa \big) \cdot n 
\text{ on }\partial\mathcal S_\kappa(t) 
\text{ and } \int_{\partial\mathcal S_\kappa (t)} u(t)\cdot\tau \, {\rm d}s = \gamma_\kappa , \quad  \text{ for all } \ \kappa  \in \{1,2, \ldots , N \}.
\end{gather}
\end{subequations}
The  circulations conditions above are important to guarantee the uniqueness of the system \eqref{div-curl}; this is related to the Hodge-De Rham theory. See for example Kato \cite{Kato}.

We now decompose the vector field $u$ in several elementary contributions which convey the influence of the vorticity, of the circulations, of the external boundary control and of the velocities of the rigid bodies. 
\subsection{Kirchhoff potentials}
\label{KP}
Consider for any $ \kappa $ in $ \{1,2, \ldots , N \}$
the functions $\xi_{ \kappa,j} (q,\cdot)=\xi_{k} (q,\cdot)$, for $j=1,2,3$ and $k=3(\kappa-1)+j$, defined by $\xi_{ \kappa,j} (q,x) :=0 $ on $\partial\mathcal{F}(q)\setminus \partial \mathcal{S}_{\kappa}$ and by $\xi_{ \kappa,j} (q,x) := e_{j}, \text{ for } j=1,2$, and $ \xi_{ \kappa,3} (q,x) := (x-h_\kappa)^\perp \text{ on }\partial\mathcal{S}_{\kappa}.$
Above $e_1$ and $e_2$ are the unit vectors of the canonical basis.

We denote by $K_{ \kappa,j} (q,\cdot)=K_{k} (q,\cdot)$ the normal trace of $\xi_{ \kappa,j} $ on $ \partial \mathcal{F}(q)$, that is:
$K_{ \kappa,j} (q,\cdot) := n \cdot \xi_{ \kappa,j} (q,\cdot) \text{ on } \partial\mathcal{F}(q)$,
where as before $n$ denotes the unit normal vector pointing outside ${\mathcal F}(q)$. 

We introduce the Kirchhoff potentials $\varphi_{ \kappa,j}(q,\cdot)=\varphi_{k}(q,\cdot)$, for $j=1,2,3$ and $k=3(\kappa-1)+j$, as the unique (up to an additive constant) solutions in $\mathcal F(q)$ of the following Neumann problem:
\begin{subequations} \label{Kir}
\begin{alignat}{3} \label{Kir1}
\Delta \varphi_{ \kappa,j} &= 0 & \quad & \text{ in } \mathcal F(q),\\ \label{Kir2}
\frac{\partial \varphi_{ \kappa,j}}{\partial n} (q,\cdot)&= K_{\kappa,j} (q,\cdot) & \quad & \text{ on }\partial\mathcal{F}(q).
\end{alignat}
\end{subequations}
We also denote 
\begin{equation} \label{Kir-gras}
\boldsymbol{K}_\kappa(q,\cdot) :=(K_{\kappa,1}(q,\cdot),K_{\kappa,2}(q,\cdot),K_{\kappa,3}(q,\cdot))^{t} \, \text{ and } \, 
\boldsymbol\varphi_\kappa(q,\cdot) :=(\varphi_{\kappa,1}(q,\cdot),\varphi_{\kappa,2}(q,\cdot),\varphi_{\kappa,3}(q,\cdot))^{t}.
\end{equation}
Following the same rules of notation as for $q$, we define the function $\varphi(q,\cdot)$ by concatenating 
into a vector of length $3N$ the functions $\boldsymbol\varphi_\kappa(q,\cdot)$, namely:
$$\varphi(q,\cdot) :=(\boldsymbol\varphi_1(q,\cdot)^t,\ldots,\boldsymbol\varphi_{N}(q,\cdot)^t)^t.$$
\subsection{Stream functions for the circulation}
To account for the velocity circulations around the solids, we introduce for each $\kappa $ in $ \{ 1, \dots, N\}$ the stream function $\psi_{\kappa}= \psi_{\kappa}(q,\cdot)$ defined on $\mathcal F(q)$ as the harmonic vector field which has circulation $\delta_{\kappa,\nu}$ around $\partial \mathcal{S}_{\nu}(q)$. More precisely, for every $q$, one can show that there exists a unique family $(C_{\kappa,\nu}(q))_{\nu \in \{1,2, \ldots , N \}}$ in $\mathbb R^N$ such that the unique solution $\psi_{\kappa}(q,\cdot)$ of the Dirichlet problem:
\begin{subequations} 
\label{stream_circulation}
\begin{alignat}{3}
\Delta \psi_{\kappa}(q,\cdot) & =0 & \quad & \text{ in } \mathcal F(q) \\
\psi_{\kappa}(q,\cdot) & = C_{\kappa,\nu}(q) & \quad & \text{ on } \partial \mathcal S_{\nu}(q), \text{ for } \nu \in \{1,2, \ldots , N \} , \\
\psi_{\kappa}(q,\cdot) & =0 & & \text{ on } \partial\Omega,
\end{alignat}
satisfies
\begin{equation} \label{circ-norma}
\int_{\partial\mathcal S_{\nu}(q)} \frac{\partial \psi_{\kappa}}{\partial n} (q,\cdot) {\rm d}s=-\delta_{\kappa,\nu}, \text{ for } \nu \in \{1,2, \ldots , N \},
\end{equation}
\end{subequations}
where $ \delta_{\nu, \kappa}$ is the Kronecker symbol.
As before, we define the concatenation into a vector of length $N$:
 $$\psi(q,\cdot):=(\psi_1(q,\cdot),\ldots,\psi_N(q,\cdot))^t .$$
\subsection{Hydrodynamic stream function}
For every bounded scalar function $\omega$ over $\mathcal{F}(q)$, there exists a unique family $(C_{\omega,\nu}(q))_{\nu \in \{1,2, \ldots , N \}}\in\mathbb R^N$ such that the unique solution $\psi_{\omega}(q,\cdot)$ in $ H^1(\mathcal{F}(q))$ of:
\begin{subequations} \label{def_hydro-stream}
\begin{alignat}{3}
\Delta \psi_{\omega}(q,\cdot) & =\omega & \quad & \text{ in } \mathcal F(q) \\
\psi_{\omega}(q,\cdot) & = C_{\omega,\nu}(q) & \quad & \text{ on } \partial \mathcal S_{\nu}(q), \text{ for } \nu \in \{1,2, \ldots , N \} , \\
\psi_{\omega}(q,\cdot) & =0 & & \text{ on } \partial\Omega,
\end{alignat}
satisfies
\begin{equation} \label{circ-hydro}
\int_{\partial\mathcal S_{\nu}(q)} \frac{\partial \psi_{\omega}}{\partial n} (q,\cdot) {\rm d}s=0, \text{ for } \nu \in \{1,2, \ldots , N \}.
\end{equation}
\end{subequations}
It is classical that $\nabla^\perp \psi_{\omega}$ has log-Lipschitz regularity (see again \cite{Kato} for instance). \par
We gather the stream functions  due to the fluid vorticity and to the circulations 
by setting 
$$\psi_{\omega,\gamma}(q,\cdot) :=\psi_{\omega}(q,\cdot) + \psi(q,\cdot)\cdot\gamma .$$
\subsection{Potential due to the external control} 
\label{sec-alpha}

With any $q$ in $\mathcal{Q}$ and $g $ in $\mathcal C$ we associate 
\begin{equation} \label{DefAAlpha}
{\alpha} := \mathcal A[q,g]  \in C^\infty (\overline{\mathcal{F}(q)};\mathbb{R}),
\end{equation}
the unique solution to the following Neumann problem:
\begin{equation}\label{pot}
\Delta {\alpha} =0\ \text{in}\  \mathcal{F}(q) \quad    \text{ and } \quad 
\partial_{n} \,\alpha=g  {1}_{\Sigma}\ \text{on}\ \partial\mathcal{F}(q) ,
\end{equation}
with zero mean on $\mathcal{F}(q)$ (recall \eqref{def-Cpasb}). 
This zero mean condition allows to determine a unique solution to the Neumann problem but plays no role in the sequel. 
\subsection{Decomposition of the velocity}
Now, by the linearity of System   \eqref{div-curl}, we see that the unique solution $u$ to \eqref{div-curl} can be decomposed into 
\begin{equation} \label{EQ_irrotational_flow}
u =  u_f + u_c,   \text{ with  }
u_f 
 := \sum_{\kappa=1}^N \nabla (\boldsymbol\varphi_{\kappa} (q,\cdot)\cdot \boldsymbol q'_{\kappa} )
+ \nabla^\perp\psi_{\omega,\gamma} (q,\cdot)    \text{ and }
u_c 
 := \nabla \alpha .
\end{equation}
%
%
%
%
%
%
%
%
\section{Reformulation of the Newton equations as a quadratic equation for the control}
\label{sec:no2}
This section is devoted to the reformulation of the solid equations in terms of the control, of the solid variables and of the vorticity. \par
To obtain this reformulation, we introduce test functions as follows.
For each integer ${\kappa}$ between $1$ and $N$, $q$ in $ \mathcal Q$, $\ell_{\kappa}^\ast$ in $ \R^2$ and $r_{k}^\ast $ in $ \R$,  we consider the following potential vector field $\mathcal F(q)$:
$$
u^\ast_{\kappa} :=\nabla (\boldsymbol\varphi_{\kappa} (q,\cdot)\cdot p^\ast_{\kappa}), \quad
\text{  where }  \quad p^\ast_{\kappa}:=(\ell^{\ast t}_{\kappa},r^\ast_{\kappa})^t .
$$
By \eqref{Kir1}, we have that 
$$u^\ast_{\kappa} \cdot n = \delta_{\kappa ,\nu} (\ell^\ast_{\kappa} + r^\ast_{\kappa}(\cdot - h_{\kappa})^\perp) \cdot n \,   \text{ on } \,   \partial\mathcal S_{\nu} (q)
\quad  \text{ and } \quad  
u^\ast_{\kappa} \cdot n =0  \,   \text{ on } \,  \partial \Omega .$$
By \eqref{EqTrans}, \eqref{EqRot}, \eqref{Kir2}, the fact that $u^\ast_{\kappa}$ is divergence-free in $\mathcal F(q)$, an integration by parts and \eqref{E1}, we have
\begin{equation} \label{EqIndividuelle}
m_{\kappa} h_\kappa'' \cdot\ell^\ast_{\kappa}+\mathcal J_{\kappa} \theta_\kappa'' r^\ast_{\kappa}
= \int_{\mathcal F(q)} \nabla \pi \cdot u^\ast_{\kappa} \, {\rm d}x  
= -\int_{\mathcal F(q)}\left(\frac{\partial u}{\partial t}+ \frac{1}{2}\nabla|u|^2 + \omega u^{\perp} \right)\cdot u^\ast_{\kappa} \, {\rm d}x .
\end{equation}
We introduce the global test function
$$
u^\ast := \sum_{1 \leqslant \kappa \leqslant N} u^\ast_{\kappa}
\quad \text{  for } \quad 
p^\ast :=  ( p^\ast_{\kappa})_{1 \leq \kappa \leq N} \in \mathbb R^{3N} , 
$$ 
and the genuine mass matrix $\mathcal{M}^g$ defined as the positive definite diagonal $3N\times 3N$ matrix
$$
\mathcal{M}^g:={\rm diag}( \mathcal{M}^g_1,\ldots, \mathcal{M}^g_N)
\quad    \text{ with }  \quad
{\mathcal M}^g_{\kappa} :=  {\rm diag}(	m_\kappa, m_\kappa , {\mathcal J}_\kappa).
$$
Summing \eqref{EqIndividuelle} over all indices ${\kappa}$ and using the decomposition \eqref{EQ_irrotational_flow}, we therefore obtain: 
\begin{multline} \label{EqSolidesGlobale} 
\int_{\mathcal F(q)}\frac{\partial u_c}{\partial t}  \cdot u^\ast\, {\rm d}x
+  \int_{\mathcal F(q)} \left( \frac{1}{2}\nabla|u_c|^2 \right) \cdot u^\ast\, {\rm d}x 
+  \int_{\mathcal F(q)}\nabla (u_f \cdot u_c ) \cdot u^\ast\, {\rm d}x - \int_{\mathcal F(q)} \omega u_c^{\perp}\cdot u^\ast \, {\rm d}x  \\
 = -{\mathcal M}^g q^{ \prime\prime} \cdot p^\ast  - \int_{\mathcal F(q)} \Big( \frac{\partial u_f}{\partial t}+\frac{1}{2}\nabla|u_f|^2 \Big) \cdot u^\ast\, {\rm d}x .
\end{multline}
We now reformulate each term in the left-hand side of \eqref{EqSolidesGlobale} and then handle the right-hand side. 
%
\begin{itemize}
\item Let us consider the first term in the left-hand side of \eqref{EqSolidesGlobale}.
By Leibniz's formula and Reynolds' transport formula, observing that the fluid domain is preserved by the vector field $u_f$, we have
\begin{eqnarray*}
\int_{\mathcal F(q)} \frac{\partial u_c}{\partial t}  \cdot u^\ast\, {\rm d}x
&=& \int_{\mathcal F(q)} \frac{\partial (u_c \cdot u^\ast ) }{\partial t}  \, {\rm d}x
- \int_{\mathcal F(q)} u_c \cdot \frac{\partial u^\ast}{\partial t}  \, {\rm d}x \\
&=& \frac{d}{dt}\left(\int_{\mathcal F(q)}u_c  \cdot u^\ast\, {\rm d}x \right) 
- \int_{\mathcal F(q)}u_f \cdot\nabla\left(u_c \cdot u^\ast\right)\, {\rm d}x
- \int_{\mathcal F(q)} u_c \cdot \frac{\partial u^\ast}{\partial t}  \, {\rm d}x .
\end{eqnarray*}
Then we integrate by parts the first two first integrals in the right hand side above and we compute the last one  by using the shape derivatives of the Kirchhoff potentials. We obtain
\begin{multline}\label{ga1}
\int_{\mathcal F(q)} \frac{\partial u_c}{\partial t}  \cdot u^\ast\, {\rm d}x
=
\frac{d}{dt}\left(\left(\int_{\mathcal S_\kappa(q)}\alpha \, \partial_n \boldsymbol\varphi_\kappa  (q,\cdot) {\rm d}s\right)_{1 \leq \kappa \leq N}  \cdot p^\ast\right) \\ 
-\left(\left(\int_{\mathcal S_\kappa(q)}\nabla\alpha\cdot\nabla\boldsymbol\varphi_\nu \, \partial_n \boldsymbol\varphi_\kappa  (q,\cdot) {\rm d}s\right)_{1 \leq \kappa \leq N}\cdot q'\right)_{1 \leq \nu \leq N} \cdot p^\ast \\-
\left(\left(\int_{\mathcal F(q)}\nabla\alpha \cdot \partial_{\boldsymbol q_\kappa} \nabla \boldsymbol\varphi_\nu  (q,\cdot) {\rm d}x \right)_{1 \leq \kappa \leq N}\cdot q'\right)_{1 \leq \nu \leq N} \cdot p^\ast.
\end{multline}
\item We now consider the second and third terms in the left-hand side of \eqref{EqSolidesGlobale}. By  integrations by parts, we obtain
\begin{gather}
\label{ga2} 
\int_{\mathcal F(q)} \left( \frac{1}{2}\nabla|u_c|^2 \right) \cdot u^\ast\, {\rm d}x
=\frac{1}{2} \left( \int_{ \partial \mathcal{S}_\kappa  (q)}  |  \nabla    \alpha |^{2} \, \partial_n \boldsymbol\varphi_\kappa(q,\cdot) \, {\rm d}s  \right)_{1 \leq \kappa \leq N}   \cdot p^\ast  , \\ 
\label{ga3} 
\int_{\mathcal F(q)}\nabla (u_f \cdot u_c ) \cdot u^\ast\, {\rm d}x
= \left( \int_{\partial\mathcal{S}_\kappa (q)}  ( \nabla  \alpha  \cdot u_f )   \,  \partial_n \boldsymbol\varphi_\kappa  (q,\cdot) \, ds \right)_{1 \leq \kappa \leq N}   \cdot p^\ast .
\end{gather}
\item Concerning the last term in the left-hand side of \eqref{EqSolidesGlobale}, we simply decompose: 
\begin{gather}
\label{ga4}
\int_{\mathcal F(q)} \omega u_c^{\perp}\cdot u^\ast \, {\rm d}x
=  \left(\int_{\mathcal F(q)}\omega \nabla^\perp \alpha \cdot\nabla \boldsymbol\varphi_{k}(q,\cdot)\,{\rm d}x \right)_{1 \leq \kappa \leq N}   \cdot p^\ast .
\end{gather}
\item We now turn to the right-hand side of \eqref{EqSolidesGlobale}.
By \cite[Theorem 1.2]{GLMS}, there exists a $C^{\infty}$ mapping which associates with $q$ in $\mathcal Q $ the $C^{\infty}$ mapping
\begin{equation*} 
\mathfrak F(q,\cdot):\R^{3N}\times \R^{3N} \times \R^N\times C^\infty (\overline{\mathcal{F}(q)}) \longrightarrow \R^{3N},
\end{equation*}
which depends only on the shape of $\mathcal{F}(q)$, such that 
\begin{equation} \label{ga5}
{\mathcal M}^g q^{ \prime\prime} \cdot p^\ast 
+ \int_{\mathcal F(q)} \left(\frac{\partial u_f}{\partial t}+\frac{1}{2}\nabla|u_f|^2\right)\cdot u^\ast\, {\rm d}x
=  -\mathfrak F(q,q',q'',\gamma,\omega)  \cdot p^\ast  .
\end{equation}
Actually thanks to the result \cite[Theorem 1.2]{GLMS}, the structure of the mapping $\mathfrak F $ can be made more precise. This will be useful in Section~\ref{sec:conclusion-r}. 
\end{itemize}
%
Now to recast Equation~\eqref{EqSolidesGlobale} into a concise form relying on \eqref{ga1}--\eqref{ga5}, we first introduce some notations. For $q$ in $ \mathcal{Q}$ and $g$ in $\mathcal{C}$, we set 
%
\begin{equation} \label{QL}
\mathfrak Q(q) [g]  := \frac{1}{2} \left( \int_{ \partial \mathcal{S}_\kappa  (q)}  |  \nabla    \alpha |^{2} \, \partial_n \boldsymbol\varphi_\kappa(q,\cdot) \, {\rm d}s  \right)_{1 \leq \kappa \leq N}  ,
\end{equation}
and recalling the notation $\alpha=\mathcal{A}[q,g]$ of Section~\ref{sec-alpha}, we set
\begin{gather} \label{QLL} 
\mathfrak L(q,q',\gamma,\omega) [g] :=  \left(\int_{\mathcal F(q)}\omega \nabla^\perp\alpha \cdot\nabla \boldsymbol\varphi_{\kappa}(q,\cdot)\,{\rm d}x \right)_{1 \leq \kappa \leq N} -\left( \int_{\partial\mathcal{S}_\kappa (q)}  ( \nabla  \alpha  \cdot u_f )   \,  \partial_n \boldsymbol\varphi_\kappa  (q,\cdot) \, {\rm d}s \right)_{1 \leq \kappa \leq N} \\
\nonumber
+\left(\left(\int_{\mathcal S_\kappa(q)}\nabla\alpha\cdot\nabla\boldsymbol\varphi_\kappa \, \partial_n \boldsymbol\varphi_\nu  (q,\cdot) {\rm d}s\right)_{1 \leq \nu \leq N}\cdot q'\right)_{1 \leq \kappa \leq N} \\
\nonumber+
\left(\left(\int_{\mathcal F(q)}\nabla\alpha \cdot \partial_{\boldsymbol q_\nu} \nabla \boldsymbol\varphi_\kappa (q,\cdot) {\rm d}x \right)_{1 \leq \nu \leq N}\cdot q'\right)_{1 \leq \kappa \leq N}.
\end{gather}
Also, of particular importance for the reformulation of the equation will be the following additional assumption on the control. \par
\ \\
{\bf Additional assumption on the control.}
In the sequel we will make the following additional assumption on the controls that we consider, in order to eliminate the time derivative from \eqref{ga1}. Let us first recall that the space $\mathcal C$ is defined in \eqref{def-Cpasb}. Now given $q$ in $\mathcal{Q}_\delta$, we define the set
\begin{equation} \label{DefCb}
\mathcal C_b(q) := \left\{  g \in   \mathcal C,\
\int_{\mathcal S_\kappa(q)}  \mathcal{A}[q,g] \, \partial_n \boldsymbol\varphi_\kappa  (q,\cdot) {\rm d}s =0\  \text{ for }  \,  1 \leq \kappa \leq N 
 \right\} .    
\end{equation}
To obtain Theorem~\ref{main}, we will consider controls $g$ in this set ${\mathcal C}_{b}$.
In this case when  $g$ is in $\mathcal C_b(q)$, by \eqref{ga1}, \eqref{ga2}, \eqref{ga3}, \eqref{ga4} and  \eqref{ga5}, the equation \eqref{EqSolidesGlobale}  now reads 
\begin{equation}
\label{pasproto}
\mathfrak Q(q) [g]
 +\mathfrak L(q,q',\gamma,\omega) [g]
= \mathfrak F (q,q',q'',\gamma,\omega) .
\end{equation}
\\
{\bf Conclusion.}
Therefore, under the assumption $g \in \mathcal C_b(q)$, the Euler equations \eqref{E1}-\eqref{E2}, 
the Newton equations \eqref{EqTrans}-\eqref{EqRot} for  $ \kappa $ in $\{1,2, \ldots , N \}$,  
 the interface condition \eqref{souslab},  and the boundary conditions \eqref{Yudo1}  with $g(t)$ in $\mathcal C_b(q(t))$ for every $t$ in $[0,T]$, 
 are equivalent to the problem
\begin{align}\label{problem}
\begin{cases}
&\mathfrak Q(q) [g]
 +\mathfrak L(q,q',\gamma,\omega) [g]
= \mathfrak F (q,q',q'',\gamma,\omega) , \\
&\partial_t \omega +\left( \sum_{\kappa=1}^N \nabla (\boldsymbol\varphi_{\kappa} (q,\cdot)\cdot \boldsymbol q'_{\kappa} )
+ \nabla^\perp\psi_{\omega,\gamma} (q,\cdot)   + \nabla \mathcal{A}[q,g] \right)\cdot\nabla\omega=0\text{ in }\mathcal{F}(q). \\
\end{cases}
\end{align}
\ \par
\begin{remark}
Three comments are in order.\ 
\begin{itemize}
\item 
Above it is understood that, in the converse way, the fluid velocity is recovered by the equation \eqref{EQ_irrotational_flow}.
That it satisfies  the Euler equations \eqref{E1} for  a pressure field in  $L^{\infty}(0,T; H^{1}({\mathcal F}(t)))$,   which is  unique  up to a function depending only on time,  follows from the second equation of \eqref{problem} and the property of the $\curl$ operator. On the other hand it follows from \eqref{EQ_irrotational_flow} that it satisfies the divergence free condition \eqref{E2},  the interface condition \eqref{souslab},  and the boundary conditions \eqref{Yudo1}  with $g(t)$ in $\mathcal C_b(q(t))$ for every $t$ in $[0,T]$. \ \par \
\item 
The system \eqref{problem}  will be useful in the proof of the existence part of  Theorem~\ref{main} in Section~\ref{sec:existence}, while \eqref{pasproto} alone  will be used in the proof of the uniqueness part of   Theorem~\ref{main} in Section~\ref{sec:conclusion-r}. \ \par \
\item 
The unknowns of the problem \eqref{problem}  are $g$ and $\omega$, and one may observe that this  system is completely coupled in the sense that $g$ and $\omega$ are involved in both equations. 
 Still we will tackle these two equations separately. In Section~\ref{sec:control} we will start by proving the existence of a solution of the first equation of  \eqref{problem}  for the unknown $g$ in terms of $\omega$ considered as a parameter. 
Then in Section~\ref{sec:existence} we will solve the second equation of \eqref{problem}  for $\omega$ with $g$  given by the solution of the first equation identified in Section~\ref{sec:control}.
 \end{itemize}
\end{remark}
%
%
%
%
%
%
%
%
%
%
%
\section{Design of a feedback control law}
\label{sec:control}
This section is devoted to the design of a control   $g$ (for the trace of normal velocity on the exterior boundary) on $[0,T]\times\Sigma$ of the form  $g  =  \mathscr C  (q,q',q'',\gamma,\omega)$,
where $\mathscr C$ is a Lipschitz function on 
$$\cup_{q\in \mathcal Q_\delta} \{  q\} \times \mathscr K  \times \mathscr{B}(q,r_\omega),$$
with values in ${\mathcal C}_{b}$ (defined in \eqref{DefCb}),
and aimed at fulfilling the first equation of \eqref{problem}. 
Recall that $\mathscr K$ is a compact subset   of  $  \R^{3N}  \times   \R^{3N}   \times   \R^{N}$ and
$\mathscr{B}(q,r_\omega)$ is defined in \eqref{lesboules}.
The second equation of \eqref{problem} will be tackled in the next section. \par
Precisely in this section we show the following.
\begin{proposition} \label{prop-design}
Let  $\delta>0$.  
There exists a finite dimensional subspace $\mathcal E\subset \mathcal C$ and for any $r_\omega>0$ and any  compact subset $\tilde{\mathscr K}$  of  $\R^{3N}    \times   \R^{N}$, there exists a locally Lipschitz mapping 
\begin{equation*}
\mathfrak{R} : \cup_{q\in\mathcal Q_\delta} \{  q\}  \times \tilde{\mathscr K} \times \mathscr{B}(q,r_\omega) \times \mathbb{R}^{3N} \longrightarrow {\mathcal E} , \quad 
(q,q',\gamma,\omega , p^\ast) \longmapsto \mathfrak{R}(q,q',\gamma,\omega , p^\ast) \in {\mathcal E} \cap \mathcal{C}_b (q),
\end{equation*}
%
such that for any $p^\ast$ in $  \mathbb{R}^{3N}$, 
$$ \mathfrak Q(q) [\mathfrak{R}(q,q',\gamma,\omega , p^\ast)]
+ \mathfrak L(q,q',\gamma,\omega) [\mathfrak{R}(q,q',\gamma,\omega , p^\ast)] = p^\ast .$$
\end{proposition}
\ \par
This proposition being granted, we will be able to design the control as follows. We set 
\begin{equation}  \label{def-tK}
\tilde{\mathscr K}:=\{(q',\gamma),\ (q',q'',\gamma)\in \mathscr K\text{ for some }q''\}.
\end{equation}
and we define $\mathscr C \in \text{Lip}( \cup_{q\in \mathcal Q_\delta} \{  q\} \times \mathscr K  \times \mathscr{B}(q,r_\omega); \mathcal E)$ by 
\begin{equation} \label{cladef}
 \mathscr C (q,q',q'',\gamma,\omega) := \mathfrak{R}(q,q',\gamma,\omega , \mathfrak F (q,q',q'',\gamma,\omega) ).
\end{equation}
The Lipschitz regularity of $\mathscr C$ follows from the boundedness and regularity  of $\mathfrak F$ mentioned above and from the regularity of $\mathfrak{R}$ given by  Proposition~\ref{prop-design}. \par
The rest of the section is devoted to the proof of Proposition~\ref{prop-design}. 
%
%
%
%
%
%
%
\subsection{A nonlinear method to  solve linear perturbations of nonlinear equations}
  \label{sec-proto}
Our strategy relies on a nonlinear method: we prove the existence of solutions to nonlinear equations of the type 
\begin{equation}
  \label{proto}
Q(X) + L(X) = Y
\end{equation}
where $L$ is linear continuous and $Q$ is a quadratic operator admitting a non trivial zero at which the differential is right-invertible. The main point is using the difference of homogeneity between $Q$ and $L$. Some other superlinear homogeneous terms could be considered in place of $Q$, the key point being to deal with the linear part of the equation as a perturbation of the nonlinear part by means of a scaling argument.
This type of strategy where one takes advantage of the nonlinearity is reminiscent of Coron's phantom tracking method and return method, see \cite{Coron-Pha}, respectively \cite{Coron}. 
Here we will only be interested in the existence of a solution to an equation of the form \eqref{proto}, which holds as a prototype for Equation~\eqref{pasproto} where the unknown is the control, so that we do not expect any uniqueness properties. 
Moreover we need to consider some parameterized version of \eqref{proto}, as the operators $Q$ and $L$ above depend on the parameters $q,q',\gamma,\omega$, and the feedback control law that we are looking for has to be robust and to depend on these parameters in a sufficiently regular manner. 
The precise result which we prove in this subsection is the following. 
\begin{proposition}
\label{lem:geoalg}
Let $d \geq 1$, 
let $(E,\|\cdot\|)$ be a  finite-dimensional normed linear space of dimension larger than $d$, $F$ a bounded metric space, 
$Q:F \times E\to \mathbb{R}^d$ a Lipschitz  map which to each $\mathfrak{p}$ in $ F$  associates  a quadratic operator $Q_\mathfrak{p}$ from $E$ to $\R^d$, and $L:F \times E\to \mathbb{R}^d$ a Lipschitz  map which to each $\mathfrak{p}$ in $ F$  associates a linear operator $L_\mathfrak{p}$ from $E$ to $\mathbb{R}^d$.
Furthermore, assume that there exists a Lipschitz map  $\mathfrak{p}\in F \mapsto\overline{X}_\mathfrak{p}$ in $ E$ satisfying for any $p$ in $F$,
$$\|\overline{X}_\mathfrak{p}\|=1 , \quad Q_\mathfrak{p}(\overline{X}_\mathfrak{p})=0,$$
and such that the family of linear operators $\left(D Q_\mathfrak{p}(\overline{X}_\mathfrak{p})\right)_{\mathfrak{p} \in F}$ admits a family of right inverses depending on $\mathfrak{p} \in F$ in a Lipschitz way.
Then there exists a locally Lipschitz mapping $R:F\times\mathbb{R}^d\to E$ such that 
$$(Q_\mathfrak{p}+L_\mathfrak{p})\circ R(\mathfrak{p},\cdot)=\mbox{Id}_{\mathbb{R}^d}.$$
\end{proposition}
To prove Proposition~\ref{lem:geoalg} we will make use of  the  following version of the inverse function theorem where the size of the neighborhood is precised with respect to a parameter. 
%
In the proof of Proposition~\ref{lem:geoalg} we will need to add a scalar parameter to  $\mathfrak{p}$, hence we introduce the notation
$\tilde{\mathfrak{p}}$ and $\tilde{F}$.
Furthermore, the space $E$ that we refer to in the Lemma below will not be quite the same as the one in Proposition~\ref{lem:geoalg}. Hence we will rather use the notation $\tilde E$.
%
%
Despite the fact that we will actually use it on a finite dimensional space, we state the result in the slightly more general setting of Banach spaces.  
\begin{lemma} \label{repl}
Let $\tilde{E}$ be a Banach space, $\tilde{F}$ a  metric space, for any $\tilde{\mathfrak{p}}$ in $ \tilde{F}$, $f_{\tilde{\mathfrak{p}}}:\tilde{E}\to \tilde{E}$ a  mapping which is $C^1$ in a neighborhood of $0$, such that the following are satisfied:
\begin{itemize}
\item[(i)] 
for any $\tilde{\mathfrak{p}}$ in $ \tilde{F}$, the linear map 
 $Df_{\tilde{\mathfrak{p}}}(0)$ is one-to-one on $\tilde{E}$,
the maps 
$(\tilde{\mathfrak{p}},x)\in \tilde{F} \times \tilde{E} \mapsto f_{\tilde{\mathfrak{p}}}(x)$ and $\tilde{\mathfrak{p}} \in \tilde{F} \mapsto Df_{\tilde{\mathfrak{p}}}(0)^{-1}$ are Lipschitz;
\item[(ii)]
there exist $r>0$ and $M>0$ such that   for all $x_1,x_2$ in $ B_{\tilde{E}}(0,r)$ and $\tilde{\mathfrak{p}}$ in $ \tilde{F}$, 
 $$\|Df_{\tilde{\mathfrak{p}}}(0)^{-1}\|_{\mathcal{L}(\tilde{E};\tilde{E})}\leq M
 \,   \text{ and } \, 
 \|Df_{\tilde{\mathfrak{p}}}(x_1)-Df_{\tilde{\mathfrak{p}}} (x_2)\|_{\mathcal{L}(\tilde{E};\tilde{E})}\leq \frac{1}{2M} .
$$
\end{itemize}
Then there exists a unique Lipschitz map
$$\tilde{\mathscr R}: \cup_{\tilde{\mathfrak{p}} \in \tilde{F}}\left(\{\tilde{\mathfrak{p}}\}\times B_{\tilde{E}} \left(f_{\tilde{\mathfrak{p}}}(0),\frac{r}{2M}\right)\right) \longrightarrow B_{\tilde{E}}(0,r),$$
such that for any $\tilde{\mathfrak{p}}$ in $\tilde{F}$, for any $y$ in $B_{\tilde{E}}\left(f_{\tilde{\mathfrak{p}}}(0),\frac{r}{2M}\right)$,
$$f_{\tilde{\mathfrak{p}}}(\tilde{\mathscr R}(\tilde{\mathfrak{p}},y))=y.$$ 
%
\end{lemma}
\begin{proof}[Proof of Lemma~\ref{repl}]
For $\tilde{\mathfrak{p}} $ in $ \tilde{F}$, $y$ in $ B_{\tilde{E}}\left(f_{\tilde{\mathfrak{p}}}(0),\frac{r}{2M}\right)$
and $x $ in $ B_{\tilde{E}}(0,r)$, we set 
\begin{align} \nonumber
g_{\tilde{\mathfrak{p}},y}(x) &:=x+Df_{\tilde{\mathfrak{p}}}(0)^{-1}\left(y-f_{\tilde{\mathfrak{p}}}(x) \right) 
\\ \label{louise} &= Df_{\tilde{\mathfrak{p}}}(0)^{-1}\left(y-f_{\tilde{\mathfrak{p}}}(0) \right) 
 - Df_{\tilde{\mathfrak{p}}}(0)^{-1}\left(  f_{\tilde{\mathfrak{p}}}(x) -f_{\tilde{\mathfrak{p}}}(0) - Df_{\tilde{\mathfrak{p}}}(0) x  \right).
\end{align}
Using \eqref{louise}, the triangle inequality, (ii) and
\begin{align}\label{eq:mvt}
f_{\tilde{\mathfrak{p}}}(x)-f_{\tilde{\mathfrak{p}}}(0)=\left(\int_0^1 Df_{\tilde{\mathfrak{p}}}(tx)\, dt\right)x, 
 \end{align}
 we  observe that
$$
g_{\tilde{\mathfrak{p}},y} (B_{\tilde{E}}(0,r)) \subset B_{\tilde{E}}(0,r).
$$
Similarly, for any $x_1,x_2$ in $ B_{\tilde{E}}(0,r)$, using again \eqref{louise} , (ii) and \eqref{eq:mvt}, we obtain: 
\begin{equation} \label{contract}
\|g_{\tilde{\mathfrak{p}},y}(x_1)-g_{\tilde{\mathfrak{p}},y}(x_2)\| \leq \frac{1}{2}\|x_1-x_2\|.
\end{equation}
Therefore, $g_{\tilde{\mathfrak{p}},y}$ is a contraction, and from the Banach fixed point theorem it follows that the mapping $g_{\tilde{\mathfrak{p}},y}$ has a unique fixed point in $B_{\tilde{E}}(0,r)$, which is also the unique solution of $f_{\tilde{\mathfrak{p}}}(\cdot)=y$. We consequently define $\tilde{\mathscr R}(\tilde{\mathfrak{p}},y)$ as this fixed point. \par

Now let $\tilde{\mathfrak{p}}_1$ and $\tilde{\mathfrak{p}}_2$ in $\tilde{F}$,
$$y_1 \in B_{\tilde{E}}\left(f_{\tilde{\mathfrak{p}}_1}(0),\frac{r}{2M}\right)
\,   \text{ and } \, 
y_2 \in B_{\tilde{E}}\left(f_{\tilde{\mathfrak{p}}_2}(0),\frac{r}{2M}\right) .
$$
By  the triangle inequality, using \eqref{contract} and
$$
\tilde{\mathscr R}(\tilde{\mathfrak{p}}_2,y_2) = g_{\tilde{\mathfrak{p}}_2,y_2}(\tilde{\mathscr R}(\tilde{\mathfrak{p}}_2,y_2)) \, 
\text{ and }  \,
\tilde{\mathscr R}(\tilde{\mathfrak{p}}_1,y_1) = g_{\tilde{\mathfrak{p}}_1,y_1}(\tilde{\mathscr R}(\tilde{\mathfrak{p}}_1,y_1)),
$$
 we obtain 
\begin{align*}
\| g_{\tilde{\mathfrak{p}}_1,y_1}(\tilde{\mathscr R}(\tilde{\mathfrak{p}}_1,y_1))-  g_{\tilde{\mathfrak{p}}_2,y_2}(\tilde{\mathscr R}(\tilde{\mathfrak{p}}_1,y_1))\| 
&\geq \|\tilde{\mathscr R}(\tilde{\mathfrak{p}}_2,y_2)-\tilde{\mathscr R}(\tilde{\mathfrak{p}}_1,y_1)\|
-\| g_{\tilde{\mathfrak{p}}_2,y_2}(\tilde{\mathscr R}(\tilde{\mathfrak{p}}_2,y_2))-  g_{\tilde{\mathfrak{p}}_2,y_2}(\tilde{\mathscr R}(\tilde{\mathfrak{p}}_1,y_1))\|  \\
& \geq \frac{1}{2}\|\tilde{\mathscr R}(\tilde{\mathfrak{p}}_2,y_2)-\tilde{\mathscr R}(\tilde{\mathfrak{p}}_1,y_1)\| .
\end{align*}
Now since the mapping $(\tilde{\mathfrak{p}},y,x)\mapsto g_{\tilde{\mathfrak{p}},y}(x)$ is Lipschitz due to (i), we deduce that $\tilde{\mathscr R}$ is Lipschitz.
This concludes the proof of Lemma~\ref{repl}.
\end{proof}
We can now start the proof of Proposition~\ref{lem:geoalg}.
\begin{proof}[Proof of Proposition~\ref{lem:geoalg}]
We begin by observing that for any $\mathfrak{p}$ in $ F$, there exist some linear isomorphisms $\varphi_\mathfrak{p}$ from $\R^d$ to 
the orthogonal of the kernel of $DQ_\mathfrak{p}(\overline{X}_\mathfrak{p})$,
such that the maps $F \ni {\mathfrak{p}}  \mapsto \varphi_{{\mathfrak{p}}}$  and  $F \ni {\mathfrak{p}} \mapsto \varphi_\mathfrak{{\mathfrak{p}}}^{-1}$ are Lipschitz and bounded (since $F$ is bounded).

The proof is of Proposition~\ref{lem:geoalg} is then based on a scaling argument. 
We introduce $\varepsilon_0>0$ such that for any $\varepsilon$ in $[0,\varepsilon_0]$, for any $\mathfrak{p}$ in $ F$: 
\begin{itemize}
\item[(a)] the linear operator 
$DQ_\mathfrak{p}(\overline{X}_\mathfrak{p})+\varepsilon L_\mathfrak{p} : E \longrightarrow E$
is right invertible, with some right inverses which are uniformly bounded as $(\varepsilon,\mathfrak{p})$ runs over $[0,\varepsilon_0]\times F$, 
\item[(b)] the linear isomorphism $\varphi_\mathfrak{p}$ also allows to select a right inverse of $DQ_\mathfrak{p}(\overline{X}_\mathfrak{p})+\varepsilon L_\mathfrak{p}$.
\end{itemize}
%
%

We will further denote $\tilde{\mathfrak{p}}:=(\varepsilon,\mathfrak{p})$ and $\tilde{F}:= [0,\varepsilon_0]\times F$.
Then we consider the mapping $f_{\tilde{\mathfrak{p}}} $ from $ \R^d  $ to $ \R^d  $ which maps $x$ in $ \R^d  $ to 
$$
f_{\tilde{\mathfrak{p}}} (x) := (Q_\mathfrak{p}
+ \varepsilon L_\mathfrak{p})(\overline{X}_\mathfrak{p}+\varphi_{{\mathfrak{p}}} (x) ) \in \R^d .$$
Our goal is to apply Lemma \ref{repl} to the mapping $f_{\tilde{\mathfrak{p}}}  $ and to the space $\tilde{E}= \R^d$. \par
First, we see that the map $(\tilde{\mathfrak{p}},x) $ in $ \tilde{F} \times  \R^d \mapsto f_{\tilde{\mathfrak{p}}} (x) $ is  Lipschitz and for any $\tilde{\mathfrak{p}} $ in $ \tilde{F}$,
the map $x $ in $ \R^d \mapsto f_{\tilde{\mathfrak{p}}} (x) $ is $C^1$ in a neighborhood $\mathcal V$ of $0$ in $\R^d$ and for any $x$ in  $\mathcal V$, 
$$
Df_{\tilde{\mathfrak{p}}} (x) 
= (DQ_\mathfrak{p}+\varepsilon L_\mathfrak{p})(\overline{X}_\mathfrak{p} + \varphi_{{\mathfrak{p}}} (x) )
\circ \varphi_{{\mathfrak{p}}} .
$$
In particular, since $\varphi_{{\mathfrak{p}}} (0)=0$,
$$
Df_{\tilde{\mathfrak{p}}} (0) =
(DQ_\mathfrak{p}+\varepsilon L_\mathfrak{p})(\overline{X}_\mathfrak{p}  ) \circ \varphi_{{\mathfrak{p}}} ,
$$
is one to one and, using (b) from the above choice of $\varepsilon_0>0$, one can see that its inverse is given by 
\begin{equation*}
Df_{\tilde{\mathfrak{p}}} (0)^{-1}
 =  \varphi_{{\mathfrak{p}}}^{-1} \circ ((DQ_\mathfrak{p}+\varepsilon L_\mathfrak{p})(\overline{X}_\mathfrak{p}))^{-1} . 
\end{equation*}
Thus the map $\tilde{\mathfrak{p}} \mapsto Df_{\tilde{\mathfrak{p}}} (0)^{-1}$ is Lipschitz and bounded as the composition of bounded Lipschitz maps.
Therefore the assumption (i) of  Lemma~\ref{repl} is satisfied. 
 
Moreover for all $x_1,x_2$ in $\mathcal V$ and $\tilde{\mathfrak{p}}\in \tilde{F}$, we have
$$Df_{\tilde{\mathfrak{p}}} (x_1)  -  Df_{\tilde{\mathfrak{p}}} (x_2) 
= \Big((DQ_\mathfrak{p}+\varepsilon L_\mathfrak{p})(\overline{X}_\mathfrak{p} + \varphi_{{\mathfrak{p}}}(x_1) )
-  (DQ_\mathfrak{p}+\varepsilon L_\mathfrak{p})(\overline{X}_\mathfrak{p} + \varphi_{{\mathfrak{p}}} (x_2) ) \Big)
\circ \varphi_{{\mathfrak{p}}} .$$
Using that the mapping $F \times E \ni (\mathfrak{p},x) \longmapsto Q_\mathfrak{p} (x) \in \mathbb{R}^d $ is Lipschitz and that  $\varphi_{{\mathfrak{p}}} $ is linear continuous, we deduce that that the assumption (ii) of  Lemma \ref{repl} is satisfied. 

Hence we can apply Lemma \ref{repl} to $f_{\tilde{\mathfrak{p}}}$ and obtain the map $\tilde{\mathscr R}$.
Since $Q_\mathfrak{p}(\overline{X}_\mathfrak{p})=0$ and $\varphi_{{\mathfrak{p}}} (0) = 0$ (by linearity),
we have $f_{\tilde{\mathfrak{p}}} (0) = \varepsilon L_\mathfrak{p} (\overline{X}_\mathfrak{p})$. 
Therefore $f_{\tilde{\mathfrak{p}}} (0) $ converges to $0$ as $\varepsilon$ converges to $0$,
uniformly in $\tilde{\mathfrak{p}}\in \tilde{F} $. 
Then reducing $\varepsilon_0>0$ again if necessary, there exists $r>0$
such that the Lipschitz mapping
$$
\mathscr R:   [0,\varepsilon_0] \times F \times B_{\R^d}(0,r) \longrightarrow E,
\quad \quad
(\varepsilon,\mathfrak{p},y)  \longmapsto
\overline{X}_\mathfrak{p} +  \tilde{\mathscr R} (\varepsilon,\mathfrak{p},y) ,
$$
satisfies for any $\varepsilon$ in $[0,\varepsilon_0]$, 
for any $\mathfrak{p}$ in $F$, for any $y$ in $\mathbb{R}^{d}$ with $|y| < r$,
\begin{align}\label{eq:Reps}
(Q_\mathfrak{p}+\varepsilon L_\mathfrak{p} ) (\mathscr R (\varepsilon,\mathfrak{p},y)) = y  . 
\end{align}
Now we define  $R: F \times \mathbb{R}^d\to E$ by setting, for any $\mathfrak{p}$ in $F$ and
any $y$ in $\mathbb{R}^d$, 
$$
R(\mathfrak{p},y):=\frac{1}{\varepsilon (y)} \mathscr R (\varepsilon (y),\mathfrak{p},\varepsilon (y)^2 y)
\  \text{ where } \  \varepsilon (y):=\min\left\{\varepsilon_0,\frac{\sqrt{r}}{(1+|y|^2)^{1/4}}\right\} .
$$
Observe that this definition makes sense since, for any $y$ in $\mathbb{R}^{d}$,  $\varepsilon (y)$ is in $(0, \varepsilon_0 ]$ and $|\varepsilon(y)^2 y| <r$.
Moreover $R$ is a locally Lipschitz mapping as composition of locally Lipschitz mappings. 
Finally, using that the mapping $Q_\mathfrak{p}$ is quadratic, that the mapping 
$L_\mathfrak{p}$ is linear, and \eqref{eq:Reps}, we obtain that 
for any $\mathfrak{p}$ in $F$,
for any $y$ in $\mathbb{R}^d$, 
\begin{align*}
(Q_\mathfrak{p}+L_\mathfrak{p})(R(\mathfrak{p},y))=\frac{1}{\varepsilon (y)^2}(Q_\mathfrak{p}+\varepsilon (y) L_\mathfrak{p})(\mathscr R (\varepsilon (y),\mathfrak{p},\varepsilon (y)^2 y))=y .
\end{align*}
This concludes the proof of Proposition~\ref{lem:geoalg}.
\end{proof}
\subsection{Restriction of the quadratic mapping $\mathfrak{Q}(q)$ and determination of a particular non-trivial zero point}

We go back to the framework of Proposition \ref{prop-design}.
We first recall that $\mathfrak Q(q)[g]$ for $q$ in $ \mathcal{Q}$ and $g$ in $ \mathcal{C}_b(q)$ was defined in \eqref{QL} with $\alpha=\mathcal{A} [q,g]$ introduced in \eqref{DefAAlpha}. Accordingly we have
\begin{equation}\label{eq:frQ}
\mathfrak Q(q)[g]= \frac{1}{2} \left( \int_{\partial\mathcal{S}_\kappa(q)}
\left | \nabla \mathcal{A} [q,g] \right|^2 \partial_n \boldsymbol\varphi_\kappa(q,\cdot)
\, {\rm d}s \right)_{\kappa=1,\ldots,N}  \in \R^{3N}.
\end{equation}
The goal of this section is to associate with the operators $\mathfrak Q(q)$ a finite-dimensional subspace ${\mathcal E} \subset {\mathcal C}_{b}$ and, for each $q$, a point in ${\mathcal E}$ which is a non trivial  zero of $\mathfrak Q(q)_{|{\mathcal E}}$ at which the derivative is right-invertible. This will allow us to apply Proposition~\ref{lem:geoalg}.  Precisely, we show the following.
\begin{proposition} \label{prop:tempo}
Let $\delta>0$. 
There exists a finite dimensional subspace $\mathcal E\subset \mathcal C$ and Lipschitz mappings
$$
q \in\mathcal{Q}_\delta  \longmapsto g_i (q,\cdot) \in \mathcal{C}_b(q) \cap \mathcal E, \,   \text{ for } \,  1 \leqslant i \leqslant (3N+1)^2 ,
$$ 
such that the following holds. Define $Q_q:\mathbb{R}^{(3N+1)^2}   \to\mathbb{R}^{3N}$
the quadratic operator which maps $X := (X_i )_{1 \leqslant i \leqslant (3N+1)^2}$ to 
\begin{equation} \label{Q-FiniteD}
 Q_q(X) := \mathfrak Q(q) \left[ \sum_{i=1}^{(3N+1)^2} X_i \, g_i(q,\cdot) \right].
\end{equation}
Then there exists a Lipschitz map $q\in\mathcal{Q}_\delta \mapsto \overline{X}_q $ in $ \R^{(3N+1)^2}$ satisfying 
$$
\|\overline{X}_q\|=1 \ \text{ and } \ Q_q(\overline{X}_q)=0,
$$
and such that $D Q_q(\overline{X}_q)$, the derivative with respect to the second argument, is right-invertible with right inverses which depend on $q$ in a bounded Lipschitz way.
\end{proposition}
%
%
%
%
%
%
To prove Proposition~\ref{prop:tempo}, we extend the analysis performed in \cite{GKS} for a single solid to the case of several solids.
In particular we will use some arguments of complex analysis and convexity which are similar to the ones already used in \cite{GKS}.
We recall that the conical hull of $A\subset \mathbb R^d$ is defined as
\begin{align*}
\text{coni}(A):=\left\{\sum_{i=1}^{k} \lambda_i a_i,\ k \in \mathbb{N}^*, \ \lambda_i \geq 0, \ a_i\in A \right\} .
\end{align*}
\begin{proof}[Proof of Proposition~\ref{prop:tempo}]
First, we recall, see  \cite[Lemma~14]{GKS}, that if $\mathcal{S}_0 \subset \Omega$ is a bounded, closed, simply connected domain of $\R^2$ with smooth boundary, which is not a disk, then
$$
\text{coni}\{(n (x),(x-h_0)^\perp \cdot n(x)), \ x \in \partial \mathcal{S}_0 \}=\mathbb{R}^3,
$$
for any $h_0$ in $\R^2$, where $n (x)$ denotes the unit normal vector to $\mathcal{S}_0$.
Therefore, taking into account the boundary conditions of the Kirchhoff potentials, see Subsection~\ref{KP}, we deduce that for any $q_0$ in $ \mathcal{Q}_\delta$,
\begin{equation} \label{conjj}
\text{coni} \left\{ \left( \partial_n \boldsymbol\varphi_\kappa(q_0,x^\kappa)\right)_{\kappa=1,\ldots,N},\ \  (x^1,\ldots,x^N) \in \partial \mathcal{S}_1(q_0) \times \ldots \times \partial \mathcal{S}_N(q_0) \right\} = \mathbb{R}^{3N}.
\end{equation}
%
%
This allows to establish the following lemma. 
\begin{lemma} \label{lem:muireg}
Fix $q_{0} $ in $ {\mathcal Q}_{\delta}$.
For $1 \leqslant i \leqslant (3N+1)^2 $ and $\kappa=1,\ldots,N$,
there exists $x_i^\kappa$ in $\partial\mathcal{S}_\kappa(q_0)$  and positive smooth mapping $\tilde\mu_i:\mathbb{R}^{3N}\to\mathbb{R}$, such that
\begin{align} \label{eq:mui0}
\sum_{i=1}^{(3N+1)^2} \tilde\mu_i(v) \left( \partial_n \boldsymbol\varphi_\kappa(q_0,x^\kappa_i) \right)_{\kappa=1,\ldots,N} = v \ \text{ for all } v \in \mathbb{R}^{3N}.
\end{align}
\end{lemma}
\begin{proof}[Proof of Lemma~\ref{lem:muireg}]
We introduce the following notations: let $\{\mathbf{b}_1,\ldots,\mathbf{b}_{3N} \}$ the canonical orthonormal basis of $\R^{3N}$,
and $\mathbf{b}_{3N+1}:=-(1,\ldots,1)$;
let $\lambda_\ell(v): =v_\ell+\sqrt{1+|v|^2}$ for $\ell=1,\ldots,3N$, and $\lambda_{3N+1}(v):=\sqrt{1+|v|^2}$.
For any $v$ in $ \mathbb{R}^{3N}$, we have
\begin{equation} \label{RRR}
v=\sum_{\ell=1}^{3N+1}\lambda_\ell(v) \, \mathbf{b}_\ell .
\end{equation}
Now, thanks to  \eqref{conjj}, we see that for some radius $r>0$ the sphere $S(0,r)$ of $\R^{3N}$ is contained in the interior of the convex hull of
$$
\left\{ \left( \partial_n \boldsymbol\varphi_\kappa(q_0,x^\kappa) \right)_{\kappa=1,\ldots,N},
\ (x^1,\ldots,x^N) \in \partial \mathcal{S}_1(q_0) \times \ldots \times \partial \mathcal{S}_N(q_0)  \right\}.
$$
For any $\ell\in\{1,\ldots, 3N+1\}$, by Carath\'eodory's theorem there exist
points $x_{(\ell-1)(3N+1)+j}^\kappa$ in $\partial\mathcal{S}_\kappa(q_0)$ for $\kappa \in \{1,\ldots,N\}$ and $j$ in $\{1,\ldots,3N+1\}$ and scalars 
$\tilde{\lambda}_{(\ell-1)(3N+1)+j} \in [0,1)$
for $j$ in $\{1,\ldots,3N+1\}$ 
 such that
\begin{align}\label{eq:cara}
r \mathbf{b}_\ell = \sum_{j=1}^{3N+1} \tilde{\lambda}_{(\ell-1)(3N+1)+j}
\left( \partial_n \boldsymbol \varphi_\kappa(q_0,x^\kappa_{(\ell-1)(3N+1)+j}) \right)_{\kappa=1,\ldots,N}.
\end{align}
We may exclude the possibility that some $\tilde{\lambda}_i$ is $0$ as follows: if for some $i$, $\tilde{\lambda}_i=0$, then we move the corresponding points $x^\kappa_i$ on another $x^\kappa_k$ for which $\tilde{\lambda}_k \neq 0$; then we split the value $\tilde{\lambda}_k$ between $\tilde{\lambda}_k$ and $\tilde{\lambda}_i$ so that no coefficient $\tilde{\lambda}_i$ vanishes. We consider this to be the case from now on. \par
Combining \eqref{RRR} with \eqref{eq:cara}, we arrive at \eqref{eq:mui0} with for any $v$ in $ \mathbb{R}^{3N}$ and for $i=1,\ldots,(3N+1)^2$,
$$\tilde{\mu}_i(v):= \frac{1}{r} \tilde{\lambda}_i\lambda_{1+ i/(3N+1)}(v),$$
where $i/(3N+1)$ denotes the quotient of the Euclidean division of $i$ by $3N+1$.
It follows that the mappings $\tilde\mu_i$ are smooth with positive  values on $ \mathbb{R}^{3N}$. This ends the proof of Lemma~\ref{lem:muireg}.
\end{proof}
%
%
Now given the points $x_{i}^\kappa$ of Lemma~\ref{lem:muireg}, let us denote
$$
x_i^\kappa(q)= R(\theta_\kappa)(x^\kappa_i-h_{\kappa,0})+h_\kappa \in \partial\mathcal{S}_\kappa(q).
$$
%
%
We consider the $3 \times 3$ and $3N \times 3N$ rotation matrices
\begin{gather*}
\mathcal{R}_\kappa(q)=\left( \begin{array}{ccc}
R(\vartheta_\kappa)  & 0\\
0 & 1 
\end{array} \right)\in\mathbb R^{3\times 3} \ \text{ for } \kappa=1,\ldots,N, \\
\mathcal{R}(q)=\left( \begin{array}{cccc}
\mathcal{R}_1(q)  & 0 & \ldots & 0\\
0 & \mathcal{R}_2(q)   & \ldots & 0\\
\vdots & \vdots & \ddots & \vdots\\
0 & 0 & \ldots & \mathcal{R}_N(q)
\end{array} \right)\in\mathbb R^{3N\times 3N}.
\end{gather*}
Recalling the definition of $\partial_n \boldsymbol\varphi_\kappa(q,\cdot)$, we find
\begin{align} \label{eq:mui}
\sum_{i=1}^{(3N+1)^2} \tilde\mu_i(\mathcal{R}(q)^{-1}v) \left( \partial_n \boldsymbol\varphi_\kappa(q,x^\kappa_i(q)) \right)_{\kappa=1,\ldots,N} =v,  \text{ for all }v\in\mathbb{R}^{3N},\ q \in \mathcal{Q}_\delta.
\end{align}
For $1 \leqslant i \leqslant (3N+1)^2 $ and $q $ in $\mathcal{Q}_\delta$,  we set 
\begin{equation} \label{def-e-mu}
 e_i (q):=\left(\partial_n \boldsymbol\varphi_1 (q, x^1_i (q)),\ldots,\partial_n \boldsymbol\varphi_N (q, x^N_i (q))\right) \ \text{ and } \ \mu_i(q,v):=\tilde\mu_i(\mathcal{R}(q)^{-1}v) ,
\end{equation}
so that, for any $q $ in $\mathcal{Q}_\delta$ and $v$ in $\mathbb{R}^{3N}$, 
\begin{align}\label{eq:muii}
\sum_{i=1}^{(3N+1)^2}\mu_i(q,v)e_i (q)=v .
\end{align}
Now the following lemma  is the adaptation of \cite[Lemma 10]{GKS} to the case of several rigid bodies. 
\begin{lemma} \label{LemBase}
Let $\kappa $ in $ \{1, \ldots, N \}$. Given $q $ in $\mathcal{Q}_\delta$, there exists a family of functions
$$
(\tilde{\alpha}_{\kappa,\varepsilon}^{i,j}(q,\cdot))_{\varepsilon \in (0,1)} 
\,  \text{ for } \,  1 \leqslant i \leqslant (3N+1)^2 \, ,  \, 1 \leqslant j \leqslant 3N+1
 \, \text{ and } \, \varepsilon \in (0,1),
$$
which are defined and harmonic in a closed neighbourhood $ \mathcal{V}_{\kappa,\varepsilon}^{i,j}$ of $\partial\mathcal{S}_\kappa(q)$,
satisfy $\partial_{n} \tilde{\alpha}_{\kappa,\varepsilon}^{i,j} (q,\cdot)=0$ on $\partial\mathcal{S}_\kappa(q)$,
and moreover, for any  $1 \leqslant i,k \leqslant (3N+1)^2 $, for any  $1 \leqslant j,l \leqslant 3N+1$, 
$$
\left|
\int_{\partial\mathcal{S}_\kappa(q)} \nabla\tilde{\alpha}_{\kappa,\varepsilon}^{i,j}(q,\cdot) \cdot \nabla\tilde{\alpha}_{\kappa,\varepsilon}^{k,l}(q,\cdot) \partial_n\boldsymbol\varphi_k(q,\cdot) {\rm d}s
- \delta_{(i,j),(k,l)} \,  \partial_n \boldsymbol\varphi_\kappa (q, x^\kappa_i (q)) \right| 
\leq \varepsilon.
$$
\end{lemma}
\begin{proof}[Proof of Lemma~\ref{LemBase}]

We deduce from the Riemann mapping theorem the existence of a conformal map
$\Psi_\kappa: \overline{\mathbb{C}} \setminus B(0,1) \to \overline{\mathbb{C}} \setminus \mathcal{S}_\kappa(q)$, where $\overline{\mathbb{C}}$ is the Riemann sphere. Classically, this conformal map is smooth up to the boundary thanks to the regularity of $\partial\mathcal{S}_\kappa(q)$.
For any smooth function $\alpha:\partial\mathcal{S}_\kappa(q)\to\mathbb{R}$, the Cauchy-Riemann relations imply that
for any $x$ in $\partial B(0,1)$, 
\begin{align*}
\begin{split}
\partial_{n} \alpha(\Psi_\kappa(x)) &= \frac{1}{\sqrt{|\text{det}(D\Psi_\kappa(x))|}} \partial_{n_B} (\alpha\circ\Psi_\kappa )(x),\\
\int_{\partial\mathcal{S}_\kappa(q)} |\nabla\alpha(x)|^2 \, \partial_n \boldsymbol\varphi_\kappa(q,x) \, {\rm d}s &= \int_{\partial B(0,1)} |\nabla\alpha(\Psi_\kappa(x))|^2 \, \partial_{n_B} \boldsymbol\varphi_\kappa(q,\Psi_\kappa(x)) \, \frac{1}{\sqrt{|\text{det}(D\Psi_\kappa(x))|}} \, {\rm d}s,
\end{split}
\end{align*}
where $n$ and $n_B$ respectively denote the normal vectors on $\partial\mathcal{S}_\kappa(q)$ and $\partial B(0,1)$. Since $\Psi_\kappa$ is invertible, we have $|\text{det}(D\Psi_\kappa(x))|>0$, for any $x$ in $\partial B(0,1)$. \par
We consider the parameterizations
$$\{c(s)=(\cos(s),\sin(s)),\ s\in [0,2\pi]\} \, \text{ of } \, \partial B(0,1)
\ \text{ and } \ 
\{ \Psi_\kappa(c(s)),\ s\in [0,2\pi]\} \, \text{ of } \, \partial \mathcal{S}_\kappa(q),$$
and the corresponding values $s_i$ such that $x_i^\kappa (q) = \Psi_\kappa(c(s_i))$, for $1 \leqslant i \leqslant (3N+1)^2$. 
We introduce a family of smooth functions $\beta_{\rho}^{i,j}: [0,2\pi] \to \mathbb{R}$ 
defined for $\rho>0$, $1 \leqslant i \leqslant (3N+1)^2 $ and $ j$ in $\{1,\ldots,3N+1\}$, satisfying:
$$\text{supp }\beta_{\rho}^{i,j}\cap\text{supp }\beta_{\rho}^{k,l}=\emptyset  \,   \text{ for }  \,  (i,j)\neq(k,l), \quad
\int_{0}^{2\pi}\beta_{\rho}^{i,j}(s)  \,  ds=0,$$
and such that, as $\rho\to0^{+}$, $\text{diam}\left(\text{supp }\beta_{\rho}^{i,j}\right)\to 0$ and 
$$
\int_{0}^{2\pi}|\beta_{\rho}^{i,j}(s)|^2 \, \partial_n \boldsymbol\varphi_\kappa(q,c(s)) \frac{1}{\sqrt{|\text{det}(D\Psi_\kappa(c(s)))|}}  \,\, ds \longrightarrow
\frac{1}{\sqrt{|\text{det}(D\Psi_\kappa(c(s_i)))|}} \partial_n \boldsymbol\varphi_\kappa (q, \Psi_\kappa(c(s_i)) ) .
$$
We denote $\hat{a}_{k,\rho}^{i,j}$ and $\hat{b}_{k,\rho}^{i,j}$ the $k$-th Fourier coefficients of the function $\beta_{\rho}^{i,j} $. \par
Now for suitable $\rho$ and $K$, one then defines $\tilde{\alpha}_{\kappa,\varepsilon}^{i,j}(q,\cdot)$ as the truncated Laurent series:
$$
\frac12 \sum_{0<k\leq K} \frac{1}{k} \left(r^k+\frac{1}{r^k}\right)( -\hat{b}_{k,\rho}^{i,j} \cos(k\theta) +\hat{a}_{k,\rho}^{i,j}  \sin(k\theta)) , $$
composed with $\Psi_\kappa^{-1}$. 
Choosing first $\rho>0$ small enough, and then $K \in \N$ large enough, it is easy to check that this family satisfies the required properties. This ends the proof of Lemma~\ref{LemBase}.
\end{proof}
The following result of approximation of
the functions $\tilde{\alpha}_{\kappa,\varepsilon}^{i,j}(q,\cdot)$ given 
by Lemma~\ref{LemBase} is close to \cite[p. 147-149]{Cortona} and \cite{OG-Addendum}. Recall the definition of $ \mathcal C$ in \eqref{def-Cpasb}. 
\begin{lemma} \label{LemmeEta}
For fixed  $k$ in $\mathbb N$, $\varepsilon>0$, and for any $\kappa$ in $\{1,\ldots,N\}$, $q$ in $\mathcal{Q}_\delta$, there exists a family of functions
$$(g_{\kappa,\eta}^{i,j}(q,\cdot))_{\eta \in (0,1)} \in  \mathcal C,  \, \text{ for } \, 1 \leqslant i \leqslant (3N+1)^2  \, \text{ and }  \,  1 \leqslant j \leqslant 3N+1 ,$$ 
with for any $\bar\kappa $ in $ \{1,\ldots,N\}$
\begin{equation} \label{del2new}
\left\|  \mathcal{A}[q,g_{\kappa,\eta}^{i,j}(q,\cdot)]
- \delta_{\kappa,\bar\kappa} \tilde{\alpha}_{\bar\kappa,\varepsilon}^{i,j}(q,\cdot) 
\right\|_{C^k(\mathcal{V}_{\bar\kappa,\varepsilon}^{i,j}\cap\overline{\mathcal{F}(q)})} \leq \eta .
\end{equation}
\end{lemma}
\begin{proof}[Proof of Lemma~\ref{LemmeEta}]
Let  $k$ in $\mathbb N$, $\varepsilon>0$,  $\kappa$ in $\{1,\ldots,N\}$, $1 \leqslant i \leqslant (3N+1)^2 $,  $1 \leqslant j \leqslant 3N+1$ and $q$ in $\mathcal{Q}_\delta$. 
We approximate $\tilde{\alpha}_{\kappa,\varepsilon}^{i,j}(q,\cdot)$ by a function defined on $\mathcal{F}(q)$ using Runge's theorem.
Namely, we first introduce a neighbourhood $V$ of $\partial\Omega\setminus\Sigma$, disjoint from $\mathcal{V}_{\bar\kappa,\varepsilon}^{i,j}(q)$ for any $\bar\kappa$ in $\{1,\ldots,N\}$.
Next we define the holomorphic function $f^{i,j}$ on the set 
$$
V\cup\left(\bigcup_{\bar\kappa=1}^N\mathcal{V}_{\bar\kappa,\varepsilon}^{i,j} \right),
$$
by
$$
f^{i,j}
= \partial_{x_{1}} \tilde{\alpha}_{\kappa,\varepsilon}^{i,j}(q,\cdot) - i \partial_{x_{2}} \tilde{\alpha}_{\kappa,\varepsilon}^{i,j}(q,\cdot) \text{ on } \mathcal{V}_{\kappa,\varepsilon}^{i,j}, 
\, \text{ and } \, 
f^{i,j}=0  \text{ on } V\cup\left(\bigcup_{\bar\kappa\neq\kappa}\mathcal{V}_{\bar\kappa,\varepsilon}^{i,j} \right).
$$
For each $\eta>0$, there exists a rational function $r_{\kappa,\eta}^{i,j}$ with one pole in each $\mathcal{S}_{\bar\kappa}(q)$ and another outside of $\overline{\Omega}$ such that
\begin{align}\label{eq:ratap}
\|r_{\kappa,\eta}^{i,j}-f^{i,j}\|_{C^k(\mathcal{V}_{\bar\kappa,\varepsilon}^{i,j}(q)\cup V)} \leq  \eta ,
\end{align}
for any  $\bar\kappa$ in $\{1,\ldots,N\}.$ 
The fact that we may take the $C^k$
 norm comes from the interior regularity of harmonic functions, enlarging a bit the neighbourhoods. \par
The function $(\text{Re}(r_{\kappa,\eta}^{i,j}),-\text{Im}(r_{\kappa,\eta}^{i,j}))$ is curl-free in $\mathcal{F}(q)$, however since $\mathcal{F}(q)$ is not a simply-connected domain, we can not directly conclude that it is a gradient, which would require it to have vanishing circulations around the solids. However, since
$(\text{Re}(f^{i,j}),-\text{Im}(f^{i,j}))$ is a gradient in $\mathcal V_{\kappa,\varepsilon}^{i,j}$ and vanishes in the neighbourhood of the other solids, we may conclude from \eqref{eq:ratap} that $(\text{Re}(r_{\kappa,\eta}^{i,j}),-\text{Im}(r_{\kappa,\eta}^{i,j}))$ has circulations of size $\mathcal{O}(\eta)$ around each solid.
Therefore, up to subtracting to $(\text{Re}(r_{\kappa,\eta}^{i,j}),-\text{Im}(r_{\kappa,\eta}^{i,j}))$ harmonic fields corresponding to these circulations, we obtain a gradient field and consequently we can define (up to a constant)
%
%
%
%
a function $\bar{\alpha}_{\kappa,\eta}^{i,j}$ which is harmonic on $\mathcal{F}(q)$ such that
$$\| \nabla \bar{\alpha}_{\kappa,\eta}^{i,j} -  \nabla \tilde{\alpha}_{\kappa,\varepsilon}^{i,j}(q,\cdot)\|_{C^k(V \cup \bigcup_{\kappa} \mathcal{V}_{\kappa,\varepsilon}^{i,j}(q))} \leq C\eta.$$
Now, by using a continuous extension operator, we may define $g(q,\cdot)$ as a function on $\partial\mathcal{F}(q)$ such that 
$$
g(q,\cdot):=\partial_n \bar{\alpha}_{\kappa,\eta}^{i,j}  \, \text{ on } \, \partial\mathcal{F}(q)\setminus \Sigma,
\ \ \int_{\partial\mathcal{F}(q)}g = 0
\ \text{ and }  \ 
\| g \|_{C^k(\Sigma)}=\mathcal{O}(\|g \|_{C^k(\partial\mathcal{F}(q)\setminus \Sigma)}).
$$
Then we introduce $\varphi$ as the solution of the Neumann problem $\Delta \varphi = 0$ in $\mathcal{F}(q)$ and $\partial_n \varphi = g $ on $\partial\mathcal{F}(q)$.
Using elliptic regularity we deduce that
$$
\|\varphi\|_{C^{k,1/2}(\mathcal{F}(q))} \leq C \eta,
$$
for some $C>0$ independent of $\eta$.
Therefore setting
$$g_{\kappa,\eta}^{i,j}(q,\cdot) := \partial_n \bar{\alpha}_{\kappa,\eta}^{i,j} - g\text{ on }\partial\mathcal{F}(q),$$ 
we obtain
$$ \mathcal{A}[q,g_{\kappa,\eta}^{i,j}(q,\cdot)] :=\bar{\alpha}_{\kappa,\eta}^{i,j}-\varphi ,$$
and this allows us to obtain \eqref{del2new} with $C \eta$ in the right-hand side, for some constant $C>0$, instead of $\eta$.
Then to conclude, we just reparameterize the family $(\bar{\alpha}_{\kappa,\eta}^{i,j})$ with respect to $\eta$. This ends the proof of Lemma~\ref{LemmeEta}.
\end{proof}
Now one proceeds as in the proof of \cite[Lemma~12]{GKS}, using a partition of unity argument, to make the above construction Lipschitz continuous with respect to $q$.
At the same time we reduce the control space to a finite dimensional subspace of $\mathcal C$.
More precisely, we have the following result.
\begin{lemma} \label{Lem10}
Let $\delta>0$ be fixed, there exists a finite dimensional subspace $\mathcal E\subset\mathcal C$ such that
for any $\nu>0$, there exist Lipschitz mappings
$$
q \in \mathcal{Q}_\delta \longmapsto \overline{g}_\kappa^{i,j} (q,\cdot)\in \mathcal{C}(q)\cap \mathcal E, \ 
\text{ for } \,  1 \leqslant i \leqslant (3N+1)^2 ,\,   1 \leq  j \leq 3N+1, \,  1 \leq \kappa \leq N,
$$
such that for any $q $ in $ \mathcal{Q}_\delta$, $i,k\in\{1,\ldots,(3N+1)^2\}$, $j,\ell\in\{1,\ldots,3N+1\}$,
\begin{equation} \label{nd1}
\left| \int_{\partial\mathcal{S}_{\kappa}(q)} \nabla \mathcal{A}[q, \overline{g}_{\bar\kappa}^{i,j}(q,\cdot) ] \cdot \nabla \mathcal{A}[q,\overline{g}_{\hat\kappa}^{k,\ell}(q,\cdot)] \, \partial_n \boldsymbol\varphi_\kappa(q,\cdot) \, {\rm d}s
- \delta_{\kappa,\bar\kappa,\hat\kappa} \delta_{(i,j),(k,\ell)} \, \partial_n \boldsymbol\varphi_\kappa (q, x^\kappa_i (q)) \right| \leq \nu ,
\end{equation}
where $\delta_{\kappa,\bar\kappa,\hat\kappa}$ in $\{0,1\}$ is zero unless $\kappa=\bar\kappa=\hat\kappa$.
\end{lemma}
\begin{proof}[Proof of Lemma~\ref{Lem10}]
Consider $q $ in $ Q_{\delta}$ and $\nu >0$.
Choosing first $\varepsilon>0$ small enough in Lemma~\ref{LemBase} and then $\eta=\eta(\varepsilon)>0$ small enough in Lemma~\ref{LemmeEta}, we may find for this $q $ in ${Q_{\delta}}$ functions $g_{\kappa,\eta}^{i,j}=g_{\kappa,\eta}^{i,j}[q]$  satisfying the properties above, and in particular such that \eqref{nd1} is valid. \par
Note that for any $q $ in $ {\mathcal{Q}_\delta}$, the unique solution $\hat{\alpha}_{\kappa,\eta}^{i,j}(\tilde{q},q,\cdot)$, up to an additive constant, to the Neumann problem
$$
\Delta_{x} \hat{\alpha}_{\kappa,\eta}^{i,j}(\tilde{q},q,x)=0\  \text{in} \   \mathcal{F}(\tilde{q}), \  \partial_{n}\hat{\alpha}_{\kappa,\eta}^{i,j}(\tilde{q},q,x)=0\  \text{on} \  \partial \mathcal{F}(\tilde{q}) \setminus \Sigma, \ 
\partial_{n} \hat{\alpha}_{\kappa,\eta}^{i,j}(\tilde{q},q,x) = g_{\kappa,\eta}^{i,j}(q,x)\  \text{on} \  \Sigma,
$$
is Lipschitz with respect to $\tilde{q} $ in $\mathcal{Q}_\delta$ (for a detailed proof using shape derivatives, see e.g. \cite{ChambrionMunnier,hp,LM}).
Therefore, if a family of functions $g_{\kappa,\eta}^{i,j}$ satisfies \eqref{nd1} at some point $q$ in ${Q_{\delta}}$, it also satisfies \eqref{nd1} (with say $2\nu$ in the right hand side) in some neighborhood of $q$. 
Due to the compactness of ${Q_{\delta}}$, since it can be covered with such neighborhoods, one can extract a finite subcover by balls $\{B(q_\ell,r_{\ell})\}_{\ell=1 \ldots N_\delta}$.
We introduce a partition of unity $\varrho_{1}, \ldots, \varrho_{N_{\delta}}$ (according to the variable $q$) adapted to this subcover. Defining
\begin{equation*}
\overline{g}_{\kappa}^{i,j}(q,\cdot) := \sum_{\ell=1}^{N_{\delta}} \varrho_{\ell}(q) \, g_{\kappa,\eta}^{i,j}[q_{\ell}](\cdot),
\end{equation*}
we can deduce an estimate like \eqref{nd1} with $C \nu$ on the right hand side, for some positive constant $C$ independent of $\nu$. It remains then to reparameterize with respect to $\nu$ to obtain \eqref{nd1} exactly. \par
Finally, the finite dimensional subspace $\mathcal E$ is then generated by $\{g_{\kappa,\eta}^{i,j}(q_i,\cdot)\}_{i=1}^{N_\delta}$ and its dimension $N_\delta$ only depends on  $\delta$. This ends the proof of Lemma~\ref{Lem10}.
\end{proof}
\ \par
Now we would like to enforce the additional condition on the control introduced in Section~\ref{sec:no2}, that is, 
that the control belongs to ${\mathcal C}_{b}$ (defined in \eqref{DefCb}).
The starting point is as follows:
for any $q $ in $ \mathcal{Q}_\delta$,  $1 \leqslant i \leqslant (3N+1)^2 $, the vectors 
$$
\left(\int_{\partial\mathcal{S}_{\kappa}(q)}
\mathcal{A}\left[q, \sum_{\bar\kappa=1}^N \overline{g}_{\bar\kappa}^{i,j}(q,\cdot) \right]
\, \partial_n \boldsymbol\varphi_\kappa(q,\cdot)\, ds\right)_{\kappa=1,\ldots,N},
$$
for $j=1,\ldots,3N+1$, are linearly dependent in $\mathbb R^{3N}$.
Therefore, there exist $\lambda^{i,j}(q) $ in $ \R$ such that
$$
\sum_{j=1}^{3N+1} \lambda^{i,j} (q) \left(\int_{\partial\mathcal{S}_{\kappa}(q)} \mathcal{A}[q, \overline{g}_{\bar\kappa}^{i,j}(q,\cdot) ]  \, \partial_n \boldsymbol\varphi_\kappa(q,\cdot)\, ds\right)_{\kappa=1,\ldots,N}=0
\ \ \text{ and } \ \ 
 \sum_{j=1}^{3N+1} \lambda^{i,j}(q)^2=1.
$$
Moreover, relying on Cramer's formula, we can manage in order that these coefficients are Lipschitz with respect to $q$. \par
Now for any $q \in \mathcal{Q}_\delta$,  for any  $1 \leqslant i \leqslant (3N+1)^2 $, we set 
\begin{equation*}
g_i(q,\cdot):=\sum_{j=1}^{3N+1}\lambda^{i,j}(q)\sum_{\kappa=1}^N\overline{g}_{\kappa}^{i,j}(q,\cdot).
\end{equation*}
Using \eqref{nd1}, up to further reducing $\nu>0$, we obtain
\begin{gather} \label{angou}
\left| \left(  \int_{\partial\mathcal{S}_\kappa(q)} \nabla \mathcal{A}[q,g_i(q,\cdot)] \cdot \nabla \mathcal{A}[q,g_j(q,\cdot)]
\partial_n\boldsymbol\varphi_\kappa(q,\cdot)\, ds \right)_{\kappa\in\{1,\ldots,N\}}
- \delta_{i,j} \, e_i (q)  \right| \leq \nu ,   \\
\nonumber
\int_{\partial\mathcal{S}_\kappa(q)}   \mathcal{A}[q,g_i(q,\cdot)] \partial_n\boldsymbol\varphi_\kappa(q,\cdot)\, ds  = 0, 
\end{gather}
where $e_i$ is defined in \eqref{def-e-mu} for $i=1,\ldots,(3N+1)^2$. Hence, $g_i(q,\cdot)\in \mathcal{C}_{b}(q)\cap \mathcal E$.
With this family of elementary controls $g_{i}$, we are finally in position to prove Proposition~\ref{prop:tempo}. \par
\ \par
For  $\delta>0$, $\nu>0$ and $(q,v) \in \mathcal{Q}_\delta \times \R^{3N}$,  we define the function
\begin{equation} \label{tiQ}
\tilde{g} (q,v,\cdot):=\sum_{i=1}^{(3N+1)^2} \sqrt{\mu_{i}(q,v)} \, {g}_{i}(q,\cdot),
\end{equation}
in $\mathcal{C}_{b}(q)\cap \mathcal E$, where we recall that the positive functions $\mu_i$ were defined in \eqref{def-e-mu}.
We then define ${\mathcal T}: \mathcal{Q}_\delta \times \R^{3N} \rightarrow \mathcal{Q}_\delta \times \R^{3N}$ by
\begin{equation*}
\mathcal{T}: \ \ (q,v) \mapsto ({\mathcal T}_{1},{\mathcal T}_{2})(q,v):= \left( q, \left(\int_{\partial\mathcal{S}_\kappa(q)} |\nabla \mathcal{A}[q,\tilde{g} (q,v,\cdot)]|^2 \partial_n \boldsymbol\varphi_\kappa(q,\cdot)\, ds\right)_{\kappa=1,\ldots,N} \right).
\end{equation*}
Recalling \eqref{eq:frQ} and $Q_q$ from \eqref{Q-FiniteD},
we may further expand
\begin{align*}
{\mathcal T}_{2}(q,v) &=
\sum_{1 \leq i,j \leq (3N+1)^2} \sqrt{\mu_i(q,v) \mu_j(q,v)} \left(  \int_{\partial\mathcal{S}_\kappa(q)} \nabla \mathcal{A}[q,g_i(q,\cdot)] \cdot \nabla \mathcal{A}[q,{g}_j(q,\cdot)] \partial_n\boldsymbol\varphi_\kappa(q,\cdot)\, ds \right)_{\kappa\in\{1,\ldots,N\}} \\&=2Q_q\left(\left(\sqrt{\mu_i(q,v)}\right)_{i=1,\ldots,(3N+1)^2}\right). 
\end{align*}
Considering ${\mathcal T}_{2}$ as a quadratic map of the variable $(\sqrt{\mu_i(q,v)})_{i=1\ldots(3N+1)^2}$ with coefficients close to $\delta_{i,j} e_{i}$, relying on \eqref{eq:muii} and \eqref{angou}, we see that for suitably small $\nu>0$ one has for any $q_{0} \in Q_{\delta}$
\begin{equation*}
\| {\mathcal T}_{2}(q_{0}, \cdot) - \mbox{Id} \|_{C^1(B(0,1))} < \frac{1}{2}.
\end{equation*}
Consequently ${\mathcal T}_{2}$ constitutes a diffeomorphism from $B(0,1)$ onto its image, which contains at least ${B}(\mathcal{T}_2(q_0,0),1/2)$. Furthermore, since $\mathcal{T}_2(q_0,0)=\mathcal{O}(\nu)$, reducing $\nu>0$ if necessary, we have that ${B}(\mathcal{T}_2(q_0,0),1/2)$ contains ${B}(0,1/4)$. Therefore, for such $\nu$, $\mathcal{T}$ is invertible at any $(q,v)\in {\mathcal{Q}_\delta}\times {B}(0,1/4)$. \par
Now we set for $1 \leqslant i \leqslant (3N+1)^2$
\begin{equation} \label{DefXBar}
\widetilde{X}_{q,i}:=\sqrt{\mu_i(q,\tilde{v})}>0
\text{ where }
(q,\tilde{v}):=\mathcal{T}^{-1}(q,0),
\ \text{ and } \
\overline{X}_q := \frac{\widetilde{X}_q}{\big\|\widetilde{X}_q\big\|}.
\end{equation}
Using \eqref{tiQ}, we find $Q_q(\overline{X}_q) = 0$.  
Moreover it is easy to check that for $1 \leqslant i \leqslant (3N+1)^2$,
\begin{align}\label{eq:niceDq}
D Q_q(\overline{X}_q)(0,\ldots,0,1,0,\ldots,0)= \frac{1}{2}\overline{X}_{q,i} \, e_i +\mathcal{O}(\nu) .
\end{align}
Hence for $\nu>0$ small enough, thanks to \eqref{eq:muii} and to the positivity of the coordinates $\overline{X}_{q,i}$, we see that $\text{Range}(D Q_q(\overline{X}_q))=\R^{3N}$. \par
Using Lemma~\ref{lem:muireg}, \eqref{def-e-mu}, \eqref{DefXBar} and the regularity of $\mathcal{T}^{-1}$, we deduce that $\overline{X}_q$ is Lipschitz with respect to $q$ and consequently $q\mapsto D Q_q(\overline{X}_q)$ is also Lipschitz. In order to apply Proposition \ref{lem:geoalg}, it remains to make a selection of right inverses of $D Q_q(\overline{X}_q)$ which are Lipschitz with respect to $q$. A possibility for that, relying on \eqref{eq:muii} is to define 
\begin{equation*}
A_{q}: \R^{3N} \longrightarrow \R^{(3N+1)^2} \ \text{ by } \ A_{q}(v)_{i}= 2\frac{\mu_i(q,v)}{\overline{X}_{i}(q,v)},
\end{equation*}
which is Lipschitz with respect to $q$ as a quotient of Lipschitz maps with positive denominator. Then due to \eqref{eq:niceDq}, $D Q_q(\overline{X}_q) \circ A_{q} = \mbox{Id}_{\R^{3N}} + O(\nu)$. It is consequently invertible in $\R^{3N}$ through a Neumann series which is consequently also Lipschitz in $q$. This allows to define unambiguously a right-inverse to $D Q_q(\overline{X}_q)$ in a Lipschitz way with respect to $q$. \par
This concludes the proof of Proposition~\ref{prop:tempo}.
\end{proof}
\subsection{Proof of Proposition~\ref{prop-design}}
Under the assumptions of Proposition~\ref{prop-design}, we first introduce ${\mathcal E}$ and the functions $g_{i}$
given by Proposition~\ref{prop:tempo}. Next we set
$$
d=3N, \quad  E=\R^{(3N+1)^2}, \quad  F=\cup_{q\in\mathcal Q_\delta} \{  q\}  \times \tilde{\mathscr K} \times \mathscr{B}(q,  r_\omega), \quad  \mathfrak{p}=(q,q',\gamma,\omega),
$$
where we recall that  $\tilde{\mathscr K} $ is defined in   \eqref{def-tK}, 
and for  $X := (X_i )_{1 \leqslant i \leqslant 3N+1)^2}$, we set
$$
Q_\mathfrak{p}(X)=  \mathfrak Q(q ) \left[ \sum_{i=1}^{(3N+1)^2} X_i g_i(q,\cdot)\right] 
\text{ and }
L_\mathfrak{p}(X)=\mathfrak L(q,q',\gamma,\omega)\left[\sum_{i=1}^{(3N+1)^2} X_i g_i(q,\cdot)\right],
$$
recalling \eqref{QLL} and using \eqref{Q-FiniteD}.
It is classical that the Kirchhoff potentials $\varphi$ and the  stream function $\psi$  are 
 $C^\infty$  as functions of $(q,x)$ on $\cup_{q\in\mathcal{Q}_\delta}\{q\}\times\mathcal{F}(q)$, see e.g. \cite{ChambrionMunnier,hp,LM}.
The Lipschitz continuity of $\mathfrak{p}\mapsto L_\mathfrak{p}$ then follows from \eqref{QLL}, the definitions given in Section~\ref{sec:no2} and the fact that $\mathfrak L$ is linear with respect to $(q',\gamma,\omega)$. \par
 %
%
Therefore, the conditions of Proposition~\ref{lem:geoalg} are satisfied, and we apply it to obtain a Lipschitz map 
$$
R=(R_{1},\ldots,R_{(3N+1)^2}):F \times \R^{3N} \longrightarrow {\mathcal E},
$$
such that
$$
\mathfrak Q(q) \left[\sum_{i=1}^{(3N+1)^2} R_i(q,q',\gamma,\omega,\cdot) g_i(q,\cdot)\right]
+ \mathfrak L(q,q',\gamma,\omega) \left[\sum_{i=1}^{(3N+1)^2} R_i(q,q',\gamma,\omega,\cdot) g_i(q,\cdot)\right]
=\text{Id}_{\mathbb{R}^{3N}}.
$$
Finally one then sets
$$
\mathfrak{R}(q,q',\gamma,\omega):=\sum_{i=1}^{(3N+1)^2} R_i(q,q',\gamma,\omega,\cdot) g_i(q,\cdot),
$$
to conclude the proof of Proposition~\ref{prop-design}. \qed
\section{Proof of the existence part of Theorem~\ref{main}}
\label{sec:existence}
In this section, we prove Theorem~\ref{main}. 
Let $\delta >0$ and $\mathcal E$ be a  finite dimensional subspace of $\mathcal C$ as given by 
Proposition~\ref{prop-design}. 
Let $T>0$, $r_\omega>0$
and $\mathscr K$ be a compact subset of  $  \R^{3N}  \times   \R^{3N}   \times   \R^{N}$. 
Let 
$
\mathscr C
$
be given by   \eqref{cladef}. 
Let $q $ in $C^{2} ([0,T] ; \mathcal Q_\delta )$ and  $\gamma$ in $\R^{N}$ such that 
 for any  $t $ in $ [0,T]$, the triple $(q'(t),q''(t),\gamma) $ is in $ \mathscr{K}$. 
Let   $\omega_0$ in $L^\infty ( \mathcal{F} (q (0))$ such that 
 \begin{equation}
 \label{did}
\|\omega_0\|_{L^\infty(\mathcal{F}_0)}\leq r_\omega. 
 \end{equation}
To prove the existence part of Theorem~\ref{main} (i.e. the first item) 
we look for a 
  velocity field  $u $ in $   LL(T) $ with  $\curl u (0,\cdot) = \omega_0$ and for any
 $t$ in $[0,T]$, 
 $$\omega (t,\cdot) := \curl u (t,\cdot) \in  \mathscr{B}(q(t),r_\omega),$$ (recall the definition  in  \eqref{lesboules}), 
satisfying, for $t$ in $[0,T]$, the equations   \eqref{E1}-\eqref{E2}, 
  \eqref{EqTrans}-\eqref{EqRot} for  $ \kappa $ in $ \{1,2, \ldots , N \}$,  
 \eqref{souslab}, 
 \eqref{Yudo1} with $g$ given by 
 \eqref{pasmixe}, \eqref{Yudo2}  and  \eqref{eq:circs}.
 
 \subsection{Reduction to a fixed point problem for the vorticity}
 Following the analysis of Section~\ref{sec:no2}, we are going to look for a vorticity
 $\omega$ solution to the 
 second equation of System \eqref{problem} with $g$ given by 
 \eqref{pasmixe} and $\mathscr C $ by \eqref{cladef}, i.e. 
 \begin{equation}
 \label{sesese}
 g(t)= \mathscr C (q(t),q'(t),q''(t), \gamma,\omega (t,\cdot)).
 \end{equation}
Once  $\omega$  is determined, with this choice of $g$,
 the first equation of System \eqref{problem} is satisfied thanks to Proposition~\ref{prop-design}, and 
 according to Section~\ref{sec:no2}, this entails that 
the fluid velocity $u$ given by \eqref{EQ_irrotational_flow}
satisfies   \eqref{E1}-\eqref{E2}, 
 \eqref{EqTrans}-\eqref{EqRot} for  $ \kappa \in \{1,2, \ldots , N \}$,  
 \eqref{souslab},   \eqref{Yudo1}  with $g$ given by  \eqref{sesese},  and  \eqref{eq:circs}.
 
Hence we look for a solution $\omega$ of the second equation of System \eqref{problem} such that $\omega (t,\cdot) $ is in $ \mathscr{B}(q(t),r_\omega)$ for any $t$ in $[0,T]$, 
and satisfying the condition  \eqref{Yudo2}   on the entering vorticity  and the initial condition  $\omega (0,\cdot) = \omega_0$. 
We use a fixed point argument.
More precisely we look for 
  a fixed point of a mapping which maps a vorticity $\omega$ to 
 the solution $\tilde\omega$ of the transport equation:
\begin{align}\label{fpproblem}
\begin{split}
&(\partial_t +
U\cdot\nabla)\tilde\omega=0\text{ in }\mathcal{F}(q),\text{ for }t\in[0,T],\\
&\tilde\omega =0 \text{ on }\Sigma_- ,\text{ for }t\in[0,T],
\\
&\tilde\omega(0)=\omega_0 ,
\end{split}
\end{align}
where $U$ in $LL(T)$  is the following vector field associated with $\omega$: 
\begin{equation*}
U :=\sum_{\kappa=1}^N \nabla (\boldsymbol\varphi_{\kappa} (q,\cdot)\cdot \boldsymbol q'_{\kappa} )
+ \nabla^\perp\psi_{\omega,\gamma} (q,\cdot)   + \nabla   \mathcal{A}[q,\mathscr C  (q,q',q'',\gamma,\omega)] ,
\end{equation*}
cf. Section \ref{sec:no}. 
We will rely on
 the Schauder fixed point theorem which asserts as we recall that if  $\mathcal{B}$ is a nonempty convex closed subset of a 
normed space $\mathcal{X}$ and $\mathscr{F}:\mathcal{B}\mapsto\mathcal{B}$ is a continuous mapping 
 such that $\mathscr F(\mathcal{B})$ is contained in a compact subset of $ \mathcal{B}$, then
 $\mathscr{F}$  has a fixed point.

\subsection{Definition of an appropriate operator}

Let us set the functional setting. We denote
$$\mathcal{X}:=L^\infty  \left(\cup_{t \in (0,T)}  \, {t}  \times \mathcal{F} (q(t) )\right)
  \text{ and } \mathcal{B}:=\{\omega\in\mathcal{X}:\ \|\omega\|_\mathcal{X}\leq \|\omega_0\|_{L^\infty(\mathcal{F}_0)}\}.
$$
Observe that by  \eqref{did}, a vorticity $\omega$ in $\mathcal{B}$  satisfies the condition that $\omega (t,\cdot) $ is in $ \mathscr{B}(q(t),r_\omega)$ for any
 $t$ in $[0,T]$.
We endow $\mathcal{X}$ with the $L^3_t(L^3_x)$ topology. We aim at defining the operator
$$\mathscr{F}:\mathcal{B}\longrightarrow\mathcal{B},$$
which with $\omega$ in $\mathcal{B}$ associates a solution $\tilde\omega$ of \eqref{fpproblem}. 

One difficulty is that, since the vector field $U$
is not tangent on $\Sigma$ due to the presence of the control ${\mathcal A}$, this solution $\tilde\omega$ is not properly defined. To define it appropriately, we first introduce $\widetilde{\Omega}$ an open set containing $\overline{\Omega}$, and the set 
\begin{equation*}
B \Omega_{\delta} := \{ x \in \overline{\Omega}, \ d(x,\partial \Omega) \leq \delta \},   \text{ see Figure \ref{htht}}.
\end{equation*}
In connection with these domains we consider the spaces of functions with derivatives in log-Lipschitz:
\begin{multline*}
W_{\delta} := \{ f \in C^{1}(B \Omega_{\delta}), \ / \ \nabla f \in \textrm{log-Lip} (B \Omega_{\delta}) \} \,   \text{ and }
\\
\widetilde{W} := \{ f \in C^{1}(B\Omega_{\delta} \cup(\overline{\widetilde{\Omega}}\setminus \Omega )  ), \ / \ \nabla f \in \textrm{log-Lip} (B\Omega_{\delta} \cup(\overline{\widetilde{\Omega}}\setminus \Omega )  ) \}.
\end{multline*}
According to \cite[\textsection 4.6, p. 194]{Stein} there exists 
 a continuous linear extension operator $\pi$ which maps functions defined on the set $B\Omega_{\delta}$ to functions defined in the set 
$$B\Omega_{\delta} \cup(\overline{\widetilde{\Omega}}\setminus \Omega )  ,$$
 which continuously maps the space $W_{\delta}$ into  the space $\widetilde{W}$. 
Multiplying $\pi$ by a smooth cutoff function with value $1$ in a neighborhood of $\overline{\Omega}$ and $0$ in a neighborhood of $\partial\widetilde{\Omega}$, we can assume that $\pi(f)$ is compactly supported in $\overline{\widetilde{\Omega}}$. Now we extend $U$ as follows. Due to \eqref{def-Cpasb} and the definitions in Section~\ref{sec:no}, we have
\begin{equation*}
\forall t >0, \ \ \int_{\partial \Omega} U(t,\cdot) \cdot n \, {\rm d}s =0.
\end{equation*}
Consequently we may introduce, up to an additive constant, the stream function $\Psi_{U}$ associated with $U$:
\begin{equation*}
U(t,x) = \nabla^\perp \Psi_{U}(t,x) \text{ in } B \Omega_{\delta} \text{ for } t>0.
\end{equation*}
For $t$ in $[0,T]$, 
 we define, similarly to $\mathcal F(t)$, 
\begin{equation*}
\widetilde{\mathcal F}(t) := \tilde{\Omega} \setminus \bigcup_{ \kappa \in \{1,2, \ldots , N \} } \, {\mathcal S_\kappa}(t). 
\end{equation*}

Then we extend $U$ as follows: for $t>0$, we define $\widetilde{U}(t,\cdot)$ in $\widetilde{\mathcal F}(t)$  by
\begin{equation*}
\widetilde{U}(t,x) = \left\{ \begin{array}{l}
U(t,x) \text{ if } x \in {\mathcal F}(t), \\
\nabla^\perp \pi(\Psi_{U})  \text{ if } \in \widetilde{\Omega} \setminus \Omega.
\end{array} \right.
\end{equation*}
Similarly to $LL(T)$ we define  the space $\widetilde{\text{LL}}(T)$
 of log-Lipschitz vector fields
 on $\widetilde{\mathcal F} (t) $ which 
 is defined via its norm
$$\|f\|_{\widetilde{\text{LL}}(T)} :=\|f\|_{L^\infty(\cup_{t \in (0,T)}  \, {t}  \times \widetilde{\mathcal F} (q(t) ))}+\sup_{t\in[0,T]} \, \sup_{x,y \in  \widetilde{\mathcal F} (t), \, 
x\neq y}\,  \frac{|f(t,x)-f(t,y)|}{|x-y|(1+\text{ln}^{-}(|x-y|))}.$$
Then the vector field $\widetilde{U}$ is divergence-free on $\cup_{t \in (0,T)}  \, {t}  \times \mathcal{F} (q(t))$ 
and, since the extension operator $\pi$  preserves such modulus of continuity, 
 $\widetilde{U}$ is in $\widetilde{\text{LL}}(T)$.
\begin{figure}[ht] 
	\begin{center}
		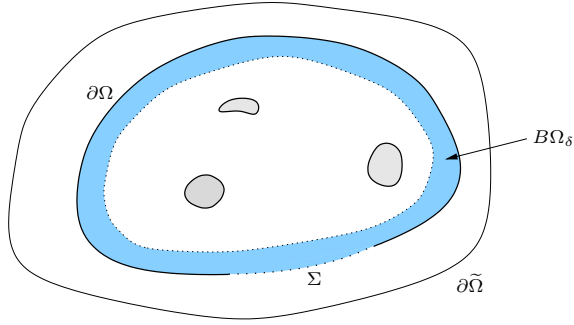
	\end{center}
    \caption{Extension of the domain}
    \label{htht}
\end{figure}

The advantage of $\widetilde{U}$ is that we may unambiguously define the flow $\Phi^{\widetilde{U}}(s,t,x)$ associated with $\widetilde{U}$. Precisely, we define the flow as the map 
\begin{equation} \label{WX1}
(s,t,x) \in [0,T] \times  [0,T]  \times \widetilde{\Omega} 
\mapsto \Phi^{\widetilde{U}}(s,t,x) \in  \widetilde{\Omega}  ,
\end{equation}
such that
\begin{equation} \label{WX2}
\partial_{s} \Phi^{\widetilde{U}}(s,t,x) = \widetilde{U}(s,\Phi^{\widetilde{U}}(s,t,x)) \ \text{ and }
\ \Phi^{\widetilde{U}}(t,t,x) =x.
\end{equation}
Now given $t>0$ and $x $ in $ \overline{{\mathcal F}(t)}$, recalling \eqref{Yudo2}, we set
\begin{equation*}
\tilde{\omega}(t,x) = \left\{ \begin{array}{l}
0 \text{ if there exists } s \in [0,T] \text{ s.t. } \ \Phi^{\widetilde{U}}(s,t,x) \in \Sigma^-, \\
\omega_{0}(\Phi^{\widetilde{U}}(0,t,x)) \text{ otherwise,}
\end{array} \right.
\end{equation*}
and we define 
$$\mathscr{F}(\omega)=\tilde{\omega}.$$
 It is indeed easy to check that when  there is no $s $ in $ [0,T]$ such that $\Phi^{\widetilde{U}}(s,t,x) $ in $ \Sigma^-$, it means that $\Phi^{\widetilde{U}}(s,t,x) $ in $ \overline{\mathcal F(s)}$ for all $s $ in $ [0,t]$, except for the negligible set for which $\Phi^{\widetilde{U}}(s,t,x) $ in $ \partial \Sigma^-$ for some $s $ in $ [0,t)$. \par
Now it is straightforward that $\mathscr F(\mathcal{B})\subset \mathcal{B}$.
To prove that  $\mathscr F$ has a fixed point by Schauder's theorem, it remains to show that $\mathscr F$ is continuous and that $\mathscr{F}(\mathcal{B})$ is relatively compact (with respect to the $L^3_t(L^3_x)$ topology).

\subsection{Continuity}
We will make use of the following result, see e.g. \cite[Lemma~1]{GS-Uniq}.
\begin{lemma}\label{ldecestb}
There exists $C=C(\mathcal{Q}_\delta)>0$ such that for any $q$ in $\mathcal{Q}_\delta$, for any $u:\mathcal{F}(q)\to\mathbb{R}^2$, for any $p\geq 2$,  there holds, with the convention
$\| f \|_{W^{1-1/q,q}(\partial {\mathcal F})} := \inf \{ \| \overline{f} \|_{W^{1,q}({\mathcal F})}, \ \overline{f} \in W^{1,q}( {\mathcal F}) \ \text{ and } \ \overline{f}_{|\partial {\mathcal F}} = f \}$:
\begin{align*}
\|u\|_{W^{1,p}(\mathcal{F}(q))}\leq  &Cp \left( \|\div u\|_{L^{p}(\mathcal{F}(q))}+\|\curl u\|_{L^{p}(\mathcal{F}(q))}\right)
\\ &\quad + C\left( \|u\cdot n\|_{W^{1-1/p,p}(\partial\mathcal{F}(q))}+\sum_{\kappa=1}^N\left|\int_{\partial\mathcal{S}_\kappa(q)} u\cdot\tau \, {\rm d}s \right|\right) .
\end{align*}
\end{lemma}
Then the following statement follows by using the classical Yudovich argument (see \cite[Lemma 2.2]{Yudo2}):
\begin{lemma}\label{ldecest}
There exists $C=C(\mathcal{Q}_\delta)>0$ such that for any $q$ in $\mathcal{Q}_\delta$, for any $u:\mathcal{F}(q)\to\mathbb{R}^2$ with $\div \, u =0$, there holds
$$\|u\|_{\textrm{log-Lip}(\mathcal{F}(q))}\leq C \left(  \|\curl u\|_{L^{\infty}(\mathcal{F}(q))} 
+ \| u \cdot n\|_{C^{1,1/2}(\partial\mathcal{F}(q))}+\sum_{\kappa=1}^N\left|\int_{\partial\mathcal{S}_\kappa(q)} u\cdot\tau \, {\rm d}s \right|\right).$$
\end{lemma}

Now let $(\omega_n)_{n\geq 1}\in \mathcal{B}^{\mathbb N}$, converging to some $\omega$ in $\mathcal{B}$ with respect to  the $L^3_t(L^3_x)$ topology. We set
\begin{align*}
U_n:=&\sum_{\kappa=1}^N \nabla (\boldsymbol\varphi_{\kappa} (q,\cdot)\cdot \boldsymbol q'_{\kappa} )
+ \nabla^\perp\psi_{\omega_n,\gamma} (q,\cdot)   + \nabla  \mathcal{A}[q,\mathscr C  (q,q',q'',\gamma,\omega_n)]
,\text{ for }n\geq 1,\\
U:=&\sum_{\kappa=1}^N \nabla (\boldsymbol\varphi_{\kappa} (q,\cdot)\cdot \boldsymbol q'_{\kappa} )
+ \nabla^\perp\psi_{\omega,\gamma} (q,\cdot)   + \nabla
 \mathcal{A}[q,\mathscr C  (q,q',q'',\gamma,\omega)] ,
\end{align*}
and correspondingly $\widetilde{U}_{n}$ and $\widetilde{U}$  as above.
Using Lemma \ref{ldecestb} 
it follows that $\widetilde{U}$ and $\widetilde{U}_n$ are uniformly bounded in $\widetilde{\text{LL}}(T) $, therefore we may define their flows $\Phi^{\widetilde{U}}$ and $\Phi^{\widetilde{U}_{n}}$, in a unique way, as in  \eqref{WX1}- \eqref{WX2}. 

Furthermore, using again Lemma \ref{ldecestb} with $p=3$ and Sobolev embeddings, the continuity of $\mathscr C$, that can be seen from \eqref{cladef},
we also get that
$\widetilde{U}_n \to \widetilde{U}$ in $L^3( (0,T) ; L^\infty ( \tilde{\mathcal F} (t) ))$, with a slight abuse of notations, 
which together with the uniform boundedness of the log-Lipschitz norms implies the uniform convergence of the flows $\Phi^{\widetilde{U}_{n}}  \to  \Phi^{\widetilde{U}}$, see for instance 
\cite[Proposition 3.9]{BCD}.
 This allows us to conclude that
$$\mathscr{F}(\omega_n) \to \mathscr{F}(\omega)\text{ in the }L^3_t(L^3_x)\text{ topology,}$$
by Lebesgue's dominated convergence theorem. For that, we first assume that $\omega_{0}$ is continuous by approximation. Then the almost everywhere convergence is obtained as follows. Given $t>0$ and $x $ in $ {\mathcal F}(t)$, there are three possibilities:\ \par \
\begin{itemize}
\item Either for all $s $ in $ [0,t]$, $\Phi(s,t,x) $ is in $ \Omega$, and then this is true for $\Phi_{n}(s,t,x) $ in $ \Omega$ if $n$ is large enough and we conclude by continuity of $\omega_{0}$;
\ \par \ \item Or for some $s $ in $ (0,t]$, $\Phi(s,t,x) $ is in $ \Sigma^-$, and then this is true for $\Phi_{n}(s'_{n},t,x) $ in $\Omega$ for $n$ large enough and some $s'_{n}\in [0,t]$, because for times $s'$ just before $s$, one has $\Phi(s',t,x) \in \overline{\widetilde{\Omega}} \setminus \Omega$,
\ \par \ \item Or $\Phi(s,t,x) $ is in $\Sigma^-$ exactly for $s=0$ or else for some $s \in (0,t]$, $\Phi(s,t,x) $ is in $\partial \Sigma^-$, but the corresponding set is negligible.
\end{itemize}
The general case when $\omega_0$ is not continuous then follows from an approximation procedure.
\subsection{Relative compactness}

Since we have a uniform bound on the log-Lipschitz norm of $u$ for $\omega\in\mathcal{B}$, it follows that  we have uniform H\"older estimates for the flow $\Phi$. We can then conclude the relative compactness of $\mathscr{F}(\mathcal{B})$ with respect to the $L^3_t(L^3_x)$ topology by using the following lemma.
\begin{lemma}
Let $C>0$, $\beta$ in $(0,1)$ and  $\omega_0$ in $ L^\infty(\mathcal{F}_0)$. Then the set
$$A(\omega_0) :=\left\{\omega_0\circ\Psi,\ \Psi\in C^\beta(\cup_{t \in (0,T)}  \, {t}  \times \mathcal{F} (q(t) );\mathcal{F}_0)\text{ measure-preserving with }\|\Psi\|_{C^\beta}\leq C \right\}$$
is relatively compact with respect to the $L^3_t(L^3_x)$ topology.
\end{lemma}
\begin{proof}
This is a slight adaptation of Lemma 12 from \cite{GS-Uniq}: 
one proves that $A(\omega_0)$ is totally bounded by approximating $\omega_0$ by a continuous vorticity and using Ascoli's theorem.
\end{proof}

\subsection{Conclusion}

Schauder's fixed point theorem then implies that $\mathscr{F}$ has a fixed point in $\mathcal{B}$. 
Using Lemma \ref{ldecestb}, one may in fact also deduce that the associated fluid velocity field is in $C^0([0,T];W^{1,p}(\mathcal{F}(t)))$, for all $p\in [1,+\infty)$.
This concludes the proof of the first part of Theorem~\ref{main}.
%
%
%
%
%
%
%
%
%
%
%
\section{Proof of the uniqueness part of   Theorem~\ref{main}}
\label{sec:conclusion-r}
This section is devoted to the uniqueness part of   Theorem~\ref{main}. 
Let $\delta >0$. We consider the finite dimensional subspace   $\mathcal E$   of $\mathcal C$  
given by Proposition~\ref{prop-design}, as well as 
 $T>0$, $r_\omega>0$,
 a compact subset $\mathscr K$  of  $  \R^{3N}  \times   \R^{3N}   \times   \R^{N}$, and the control law 
$$
\mathscr C \in \text{Lip}( \cup_{q\in \mathcal Q_\delta} \{  q\} \times \mathscr K  \times \mathscr{B}(q,r_\omega)
; \mathcal E),
$$
given by \eqref{cladef}. 
We also consider 
a trajectory 
$q $ in $C^{2} ([0,T] ; \mathcal Q_\delta )$ and $\gamma$ in $\R^{N}$ such that 
 for any 
$t $ in $ [0,T]$,   the triple  $(q'(t),q''(t),\gamma) $ is in $ \mathscr{K}$,
 $$(\tilde{q},\tilde{u}) \in C^{2} ([0,T] ; \mathcal Q_\delta ) \times [LL(T)\cap C^0([0,T];W^{1,p}(\mathcal{F}(t)))],\ \text{for all } p\in [1,+\infty),$$ 
 and  $\tilde{\gamma}$ in $\R^{N}$ such that 
 for any 
$t $ in $ [0,T]$,   the triple  $(\tilde{q}'(t),\tilde{q}''(t),\tilde{\gamma}) $ is in $ \mathscr{K}$ and $\curl \tilde{u} (t,\cdot) $ is in $ \mathscr{B}(\tilde{q}(t),r_\omega)$.
We assume that $(\tilde{q},\tilde{u}) $ satisfies 
the Euler equations \eqref{E1}-\eqref{E2}, 
 the Newton equations \eqref{EqTrans}-\eqref{EqRot} for
$ \kappa $ in $ \{1,2, \ldots , N \}$, the interface condition \eqref{souslab}, the boundary condition \eqref{Yudo1} on the normal velocity
with  $g$ given by  \eqref{mixe}, 
the boundary condition \eqref{Yudo2} on the entering vorticity,   the circulation conditions  \eqref{eq:circs}  (with $\tilde{\gamma}$ instead of ${\gamma}$) 
and the initial conditions 
\begin{equation}
\label{doni}
   \tilde{q}(0)= q(0) \, \text{ and } \, \tilde{q}'(0) = q'(0) . 
\end{equation}

 Then it follows from the analysis performed in Section  \ref{sec:no2},   
in particular from the first equation of \eqref{problem}, applied to the solution  $(\tilde{q},\tilde{u}) $, that 
\begin{equation}
\label{sinu}
\mathfrak Q(\tilde{q}) [g]
 +\mathfrak L(\tilde{q},\tilde{q}',\tilde{\gamma},\curl \tilde{u}) [g]
= \mathfrak F (\tilde{q},\tilde{q}',\tilde{q}'',\tilde{\gamma}, \curl \tilde{u}) .
\end{equation}
Then by  \eqref{mixe}, \eqref{cladef} and  Proposition~\ref{prop-design}, we infer that $$\mathfrak Q(\tilde{q}) [g]
 +\mathfrak L(\tilde{q},\tilde{q}',\tilde{\gamma},\curl \tilde{u}) [g]=\mathfrak F (\tilde{q},\tilde{q}',q'',\tilde{\gamma},\curl \tilde{u} )$$ and therefore we arrive at 
\begin{equation}
\label{cosinu}
\mathfrak F (\tilde{q},\tilde{q}',q'',\tilde{\gamma},\curl \tilde{u})
= \mathfrak F (\tilde{q},\tilde{q}',\tilde{q}'',\tilde{\gamma},\curl \tilde{u}) .
\end{equation}

Now, thanks to \cite[Theorem 1.2]{GLMS}, the structure of the mapping $\mathfrak F $ can be made more precise. 
Indeed  there exists  definite positive $3N\times 3N$ matrices $ \mathcal{M}^{a}(q)$, depending on $q$ in $\mathcal Q $ in a $C^{\infty}$ way, and 
a $C^{\infty}$ mapping  $\tilde{\mathfrak F}$ 
such that, for any admissible $(q,q',q'',\gamma,\omega)$, the term 
$\mathfrak F(q,q',q'',\gamma,\omega) $ can be decomposed into 
\begin{equation}
\label{ga555}
\mathfrak F(q,q',q'',\gamma,\omega) 
=   \mathcal{M}^{a}(q) q'' + \widetilde{\mathfrak F} ({q},{q}',\gamma,\omega) .
\end{equation}
Then, using the decomposition  \eqref{ga555} for  both sides of 
 \eqref{cosinu}, simplifying by $\widetilde{\mathfrak F} (\tilde{q},\tilde{q}',\tilde{\gamma},\curl \tilde{u})$, using the invertibility of the matrices  $ \mathcal{M}^{a}(q)$, and integrating twice in time using the initial data \eqref{doni}, we deduce that 
$q = \tilde{q} $ on $[0,T]$ and the proof of Theorem~\ref{main} is over. 
\par
\ \par 

For sake of completeness, let us explain how the proof above can be adapted  to deal with the case where the initial positions and velocities of the rigid bodies do not match, that is if $(\tilde{q}(0),\tilde{q}'(0)) \neq (q(0),q'(0) )$, but $\tilde{q}(0)$ is sufficiently close to $q(0)$, as mentioned in the comment regarding this 
part  below the statement of Theorem~\ref{main}. If $\tilde{q}(0)$ and $q(0)$ are not close, one may first use Theorem~\ref{main} to drive $\tilde{q}$ close to $q(0)$.

In this case we replace  the control law \eqref{mixe} by  \eqref{mixe-stab} and the desired result is that  the error $q (t) - \tilde{q}(t)$ exponentially decays to $0$ as the time $t$ goes to $+\infty$. 
Indeed in this case, proceeding as above, instead of  \eqref{cosinu}, we obtain the following identity:
\begin{equation}
\label{cosinu-stab}
\mathfrak F (\tilde{q},\tilde{q}',q''+ K_P (q - \tilde{q}) + K_D   (q' - \tilde{q}'),\tilde{\gamma},\curl \tilde{u})
= \mathfrak F (\tilde{q},\tilde{q}',\tilde{q}'',\tilde{\gamma},\curl \tilde{u}) .
\end{equation}
Using the decomposition  \eqref{ga555}  and the invertibility of the matrices  $ \mathcal{M}^{a}(q)$ we deduce that the error $e:= q - \tilde{q}$ satisfies the linear differential equation
$e'' + K_P \, e + K_D \,  e'=  0 $.  
Since the matrices $K_P$ and $K_D$  are  positive definite symmetric it follows that $e(t)$ 
 exponentially decays to $0$ as the time $t$ goes to $+\infty$,  with a rate which can be made arbitrarily fast by appropriate choices of $K_P$ and $K_D$, see  \cite[Proposition 4.8]{MLS}. 
%
%

\section{Some extra comments on the issue of energy saving}
 \label{afaire}

 As mentioned in the paragraph on the energy saving in the commentary below Theorem~\ref{main}  one may wonder whether it is possible to turn on the control only when the targeted motion is not already an uncontrolled solution of the system. Since such an uncontrolled equation is characterized by the equation $\mathfrak F (q,q',q'',\gamma,\omega) =0$, see  Section~\ref{sec:no2}, this issue can be formulated as the following open problem where the targeted trajectory satisfies the uncontrolled equation at the initial time but perhaps not for positive times. 
\begin{OpenP} \label{OP}  
For any  $T>0$,  $\omega_0$ in $L^\infty ( \mathcal{F} (q (0))$, $\gamma$ in $\R^{N}$  
and $q $ in $C^{2} ([0,T] ; \mathcal Q )$ such that $\mathfrak F (q(0),q'(0),q''(0),\gamma,\omega_0) = 0$, is there a boundary control  ${g}$ in $C^{\infty} ([0,T]; \mathcal C )$ with $g(0,\cdot) = 0$ and a velocity field  $u $ in $   LL(T)\cap C^0([0,T];W^{1,p}(\mathcal{F}(t))) ,\text{ for all }p\in[1,+\infty),$ with  $\curl u (0,\cdot) = \omega_0$ such that, for $t$ in $[0,T]$,  \eqref{E1}-\eqref{E2}, 
\eqref{EqTrans}-\eqref{EqRot} for  $ \kappa $ in $ \{1,2, \ldots , N \}$,  
  \eqref{souslab},  \eqref{Yudo1},  \eqref{Yudo2} and \eqref{eq:circs} hold true.  
\end{OpenP}

Several comments are in order. \ \par \

First observe that, taking into account the decomposition \eqref{EQ_irrotational_flow}, 
 the issue stated in Open problem \ref{OP} is related to the question of whether it is possible to prescribe the initial fluid velocity rather than the initial fluid vorticity (as we actually did in the first part of Theorem~\ref{main}). 
 Since the fluid velocity in the fluid domain depends on the trace of its normal component on the boundary, it is necessary to require a compatibility condition between the initial value of the control $g$ and the initial value of the fluid velocity $u$. Indeed a positive answer to Open problem \ref{OP} would entail that for any  $T>0$,  for any log-Lipschitz vector field $u_0$ such that  $\curl u_0$ in $L^\infty ( \mathcal{F} (q (0))$, 
 \begin{gather*}
\div u_0 =0 \text{ in } \mathcal{F} (q_0), \quad u_0 \cdot n = 0  \text{ on } \partial\Omega, \quad u_0 \cdot n = \big( \theta^{\prime}_\kappa  (0) (\cdot-h_\kappa (0))^\perp + h^{ \prime}_\kappa  (0) \big) \cdot n 
\text{ on }\partial\mathcal S_\kappa(0) ,
\end{gather*}
and $q $ in $C^{2} ([0,T] ; \mathcal Q )$ such that $\mathfrak F (q(0),q'(0),q''(0),\gamma,\curl u_0) = 0$, where  
\begin{equation*} 
 \gamma:= (\gamma_\kappa )_{\kappa=1,\ldots,N}  , \quad   \text{ with }   \quad
  \gamma_\kappa = \int_{\partial\mathcal S_\kappa (0)} u_0 \cdot\tau \, {\rm d}s   , \quad  \text{ for all } \ \kappa  \in \{1,2, \ldots , N \},
\end{equation*}
 there is a boundary control  ${g}$ in $C^{\infty} ([0,T]; \mathcal C )$ with $g(0,\cdot) = 0$ and a velocity field  $u $ in $   LL(T)\cap C^0([0,T];W^{1,p}(\mathcal{F}(t))),\text{ for all } p\in[1,+\infty) ,$ with  $u (0,\cdot) = u_0$ such that, for $t$ in $[0,T]$,  \eqref{E1}-\eqref{E2}, 
\eqref{EqTrans}-\eqref{EqRot} for  $ \kappa $ in $ \{1,2, \ldots , N \}$,  
  \eqref{souslab},  \eqref{Yudo1},  \eqref{Yudo2} and \eqref{eq:circs}   hold true.  
 \ \par \

Let us also observe that if one is able to answer by a positive result to Open problem \ref{OP}, then by using the time-reversibility of the system, one can deduce the following result where the targeted trajectory is an uncontrolled solution, associated with a vanishing vorticity, at the initial and final times, that is to the result 
that for any $T>0$,   $\gamma$ in $\R^{N}$ and $q $ in $C^{2} ([0,T] ; \mathcal Q )$ such that 
$$\mathfrak F (q(0),q'(0),q''(0),\gamma,0) = 0 \,   \text{ and } \,  \mathfrak F (q(T),q'(T),q''(T),\gamma,0) = 0 ,$$ 
 there a boundary control  ${g}$ in $C^{\infty} ([0,T]; \mathcal C )$ with 
  $g(0,\cdot) =g(T,\cdot) = 0$ and a velocity field  $u $ in $LL(T)\cap C^0([0,T];W^{1,p}(\mathcal{F}(t))),\text{ for all } p\in[1,+\infty) ,$ with  $\curl u (t,\cdot) = 0$ for any
 $t$ in $[0,T]$, such that, for $t$ in $[0,T]$, 
  \eqref{E1}-\eqref{E2}, 
\eqref{EqTrans}-\eqref{EqRot} for  $ \kappa $ in $ \{1,2, \ldots , N \}$,  
  \eqref{souslab},  \eqref{Yudo1},  \eqref{Yudo2} and \eqref{eq:circs}
 hold true.  
Here we have restricted the issue to the setting of irrotational flows since it is the only case where the vorticity dynamics is under control. 
Indeed it seems difficult to reach a targeted trajectory which is at time $T>0$ an uncontrolled solution corresponding to a non vanishing given vorticity with a control vanishing at time $T$,   because the dynamics of the vorticity which remains close to the rigid bodies seems difficult to control from the external boundary. \ \par \

Inspecting the proof of Proposition~\ref{lem:geoalg} we observe that the 
mapping $R$ which is constructed there satisfies  
$R(\mathfrak{p},0) = \overline{X}_\mathfrak{p}$ (for any $\mathfrak{p}$). 
In particular since $\|\overline{X}_\mathfrak{p}\|=1$, we have
$R(\mathfrak{p},0) \neq  0$. It would be interesting to investigate alternative constructions  of similar mappings $R$ with the additional condition 
$R(\mathfrak{p},0) =  0$ since  this would entail that the corresponding mappings 
$\mathscr C (q,q',q'',\gamma,\omega)$ defined by \eqref{cladef} 
vanishes when  $\mathfrak F (q,q',q'',\gamma,\omega) =0$. 
Perhaps tools from algebraic geometry could be useful, see \cite{GMS}. \ \par \

If one looks for a control $g$ of a different form than  $g  =  \mathscr C  (q,q',q'', \gamma,\omega)$, potentially not in the set 
 $ \mathcal{C}_b (q)$, 
 one may wonder whether it is possible to take advantage of the term with  the time derivative in \eqref{ga1} 
to control the motion, with the idea to determine the control as the solution of a first order ODE in time. 
If the quadratic term does not cancel for the controls chosen in this strategy, then it is a nonlinear ODE which may lead to a blow-up in finite time. 
In our construction, because of the rigidity of harmonic functions,  it seems difficult to find controls for which the term in the parenthesis in the first term of the right hand side of \eqref{ga1} reaches arbitrary value while corresponding to a vanishing  quadratic term. 
Therefore this seems limited to the case where the targeted motion for the rigid bodies is close to an uncontrolled solution for which the right hand side of \eqref{pasproto} vanishes. However, it could be that one may start with such a control before switching to the quadratic control constructed in this paper.

\ \par
\ \par
\noindent
{\bf Acknowledgements.} The authors are partially supported by the Agence Nationale de la Recherche, Project IFSMACS, grant ANR-15-CE40-0010 
and Project SINGFLOWS grant ANR-18-CE40-0027-01. 
The last author is also partially supported by the Agence Nationale de la Recherche, Project BORDS, grant ANR-16-CE40-0027-01 and by the H2020-MSCA-ITN-2017 program, Project ConFlex, Grant ETN-765579. 
%
%
%
%
%
\def\cprime{$'$}

\end{document}